\documentclass[12pt,a4paper]{article}
\usepackage{amsfonts}

\usepackage{mathrsfs}

\usepackage{stmaryrd}
\usepackage{amsmath}
\usepackage{amssymb}
\usepackage{xcolor}
\usepackage{bbm}
\usepackage{mathrsfs}
\usepackage{graphicx}
\usepackage{lipsum}

\usepackage{enumerate}
\usepackage{theorem}
\usepackage{indentfirst}
\makeatletter
\def\tank#1{\protected@xdef\@thanks{\@thanks
        \protect\footnotetext[0]{#1}}}
\def\bigfoot{

    \@footnotetext}
\makeatother

\newcommand{\ea}{\end{array}}

\newtheorem{theorem}{Theorem}[section]
\newtheorem{hypothesis}{Hypothesis}[section]

\newtheorem{proposition}{Proposition}[section]

\newtheorem{lemma}{Lemma}[section]
\newtheorem{definition}{Definition}[section]
\newtheorem{Rem}{Remark}[section]

{\theorembodyfont{\rmfamily}
}

\newenvironment{proof}{Proof.}

\def \eref#1{\hbox{(\ref{#1})}}

\renewcommand{\d}{d}

\begin{document}
\title{{\Large \bf Large Deviation Principle for McKean-Vlasov Quasilinear Stochastic Evolution Equations}}

\author{{Wei Hong$^{a}$},~~{Shihu Li$^{b}$}\footnote{Corresponding author \newline E-mail addresses: weihong@tju.edu.cn, shihuli@jsnu.edu.cn, weiliu@jsnu.edu.cn},~~{Wei Liu$^{b,c}$}
\\
 \small $a.$ Center for Applied Mathematics, Tianjin University, Tianjin 300072, China \\
 \small $b.$ School of Mathematics and Statistics, Jiangsu Normal University, Xuzhou 221116, China\\
  \small $c.$ Research Institute of Mathematical Sciences, Jiangsu Normal University, Xuzhou 221116, China}
\date{}
\maketitle
\begin{center}
\begin{minipage}{145mm}
{\bf Abstract.} This paper is devoted to investigating the Freidlin-Wentzell's large deviation principle for a class of McKean-Vlasov quasilinear SPDEs perturbed by small multiplicative noise. We adopt the variational framework and the modified weak convergence criteria to prove the Laplace principle for McKean-Vlasov type SPDEs, which
is equivalent to the large deviation principle. Moreover, we do not assume any compactness condition of embedding in the
Gelfand triple to handle both the cases of bounded and unbounded  domains in applications. The main results can be applied to various McKean-Vlasov type SPDEs such as distribution dependent stochastic porous media type equations and  stochastic p-Laplace type equations.

\vspace{3mm} {\bf Keywords:} McKean-Vlasov SPDE; Large deviation principle;~Weak convergence method;~Porous media equation;~p-Laplace equation.


\end{minipage}
\end{center}

\section{Introduction}
The large deviation principle (LDP) is one of classical topics in the probability theory, which mainly characterizes  the asymptotic property of remote tails of a family of probability distributions
and has widespread applications in different areas such as information theory,  thermodynamics, statistics and engineering. In the case of stochastic processes, a key idea of studying the LDP is to find a deterministic path
which the diffusion is concentrated around with high probability. For this purpose, the small perturbation type  LDP (also called Freidlin-Wentzell's LDP) for stochastic differential equations in finite dimensional case is established by Freidlin and Wentzell in the pioneering work \cite{FW}. Afterwards, it has been widely investigated in the past several decades, we refer the readers to the classical monographs \cite{A,DZ,S,V1} and references therein for the detailed exposition on the background and applications of large deviation theory.

There are a great deal of literatures studying the Freidlin-Wentzell's LDP in various mathematical settings through different approaches. For example, Freidlin \cite{Fr} concerned the LDP associated with the small noise limit for a class of stochastic reaction-diffusion equations, one can see also \cite{DaZa} by Da Prato and Zabczyk or \cite{P} by Peszat for the extensions to the infinite
 dimensional diffusions or stochastic partial differential equations (SPDEs) with global Lipschitz drifts.
We  refer the reader to \cite{CR} that the authors investigated the LDP for stochastic reaction-diffusion equations with local Lipschitz reaction term. As the first large deviation result for quasilinear type SPDEs,  R\"{o}ckner et al. \cite{RWW} derived the LDP in both cases of small noise and small time for stochastic generalized porous media equations.

The above-mentioned papers all use the classical time discretization approach and the contraction principle introduced by Freidlin and Wentzell in \cite{FW}, which could be very complicated in terms of studying the LDP for infinite dimensional nonlinear stochastic dynamical systems.  Another approach of investigating the LDP is the well-known weak convergence method developed by Budhiraja, Dupuis and Ellis et al in \cite{BCD,BD,BDM,DE}, which is mainly based on the variational representation formula for measurable functionals of Brownian motion. For some relevant LDP results using weak convergence method, we refer the interested readers to the works \cite{BGJ,BM,CG,CM,DWZZ,MSS,MSZ,RZ12,RZ2,SS,XZ} and references therein for recent progress on different stochastic dynamical systems.

The main goal of this paper is to study the asymptotic property for a family of McKean-Vlasov stochastic partial differential equations (MVSPDEs) with small Gaussian noise as follows,
\begin{equation}\label{e00}
\left\{ \begin{aligned}
&dX_t^\varepsilon=A(t,X_t^\varepsilon,\mathscr{L}_{X_t^\varepsilon})dt+\sqrt{\varepsilon}B(t,X_t^\varepsilon,\mathscr{L}_{X_t^\varepsilon})dW_t,\\
&X_0^\varepsilon=x,
\end{aligned} \right.
\end{equation}
where $\varepsilon>0$ is a small parameter, $W_t$ is a cylindrical Wiener process and $\mathscr{L}_{X_t^\varepsilon}$ stands for the distribution of solution $X_t^\varepsilon$, the coefficients $A,B$ fulfill some appropriate assumptions given in the next section.

Due to its intrinsic link with nonlinear Fokker-Planck-Kolmogorov equations and many other
applications, MVS(P)DEs, also called distribution dependent S(P)DEs or mean field S(P)DEs, have attracted widespread attention in recent years. MVSDEs were initially developed by Kac \cite{K1} to investigate the Boltzman equation corresponding to the particle density
in diluted monatomic gases and to study the stochastic toy model for the
Vlasov kinetic equation of plasma. We also refer the reader to the classical paper \cite{M1} for the connections between MVSDEs and interacting particle systems.
Recently, Wang \cite{W1} derived the strong/weak existence and uniqueness of solutions to a class of MVSDEs with monotone drift and studied the Harnack type inequalities and exponential ergodicity, which can be applied to deal with the homogeneous Landau type equations. Subsequently, the well-posedness, gradient estimates
and Harnack type inequalities for MVSDEs with singular coefficients have been studied by Huang and Wang in \cite{HW1}. For more investigations on this topic, one can see \cite{BLPR,BR0,BR1,BR2,HRW,HRW1,HS,RTW,RW3} and references therein.

To the best of our knowledge, there are only  few results concerning the LDP for MVSDEs so far. Herrmann et al.~\cite{HIP} obtained the Freidlin-Wentzell's LDP  in path space with the uniform topology under the superlinear growth and coercivity
hypothesises of the drift. Dos Reis et al.~\cite{DST} also investigated the Freidlin-Wentzell's LDP in both uniform and H\"{o}lder topologies via assuming that the coefficients satisfy some extra time H\"{o}lder continuity conditions. Ren and Wang \cite{RW2} established Donsker-Varadhan type large deviations for a class of semilinear path-distribution dependent SPDEs by using exponential equivalence arguments.  The main method in \cite{HIP} and \cite{DST} is to use the time discretization and certain approximating technique to get the small perturbation type LDP. As we mentioned above, this classical techniques are quite complicated to extend to study nonlinear SPDEs by showing the exponential estimates and tightness. In order to overcome this difficulty, the corresponding weak convergence criteria with respect to distribution dependent case was established in the recent work \cite{LSZZ}, where the authors proved the Freidlin-Wentzell's LDP for a class of MVSDEs under the similar framework as in \cite{DST} but with small L\'{e}vy jumps. Besides the LDP, the authors in \cite{LSZZ} also establish the criteria to investigate the moderate deviation theory.

In this paper, based on this modified weak convergence criteria in \cite{LSZZ}, we aim to study the Freidlin-Wentzell's LDP for a class of nonlinear MVSPDEs (\ref{e00}) driven by multiplicative Gaussian noise, and we use the variational framework that has been constructed in our recent work \cite[Section 2]{HL} to study the existence and uniqueness of solutions for a family nonlinear MVSPDEs. We need to point out that in comparison to \cite{LSZZ}, we now work on the MVSDEs in infinite dimensional case using the Gelfand triple (see Section 2.1 below) since we want to cover some quasilinear type MVSPDEs. Thus we need to derive some apriori estimates of solutions involving different spaces, which is quite different to the finite dimensional case. To the best of our knowledge, there is no LDP result for McKean-Vlasov quasilinear SPDEs in the literature. Another point is that we want to drop the compactness assumption of Gelfand triple in some existing works (cf.~\cite{L1,RZ2}) concerning the classcial variational framework (without distribution dependence), see \cite[Section 3]{RZ2} or \cite[Section 2]{L1} for the details.
To overcome such difficulty, some time discretization techniques mainly inspired by \cite{CM} and \cite{K2} are also employed. Based on this generalized variational framework, our main results on the LDP are applicable to several McKean-Vlasov quasilinear SPDEs such as distribution dependent stochastic porous media type equations and stochastic p-Laplace type equations.

The rest of this article is organized as follows. In Section 2, we present the mathematical framework and recall some basic theory of the LDP. Then we formulate the main results of the present paper. Section 3 is devoted to proving our main results, and in Section 4, some concrete MVSPDE models are given to illustrate the applications of the main results.

\section{Main Results}
\setcounter{equation}{0}
 \setcounter{definition}{0}
 In this section, we first introduce the mathematical framework for MVSPDEs and recall the definitions of large deviation principle and Laplace principle with their relations. Furthermore, we state the weak convergence criteria with respect to the case of distribution dependence and then formulate the main results of this work.

\subsection{Mathematical framework}
Now we present the generalized variational framework for SPDE with distribution  dependent coefficients.

Let us denote by $(U,\langle\cdot,\cdot\rangle_U), (H, \langle\cdot,\cdot\rangle_H)$ some separable Hilbert spaces and $H^*$ the dual space of $H$. Let $V$ be a reflexive Banach space such that $V\subset H$ is continuous and dense. We identify $H$ with its dual space by means of the Riesz isomorphism, then one can get the following Gelfand triple
$$V\subset H(\cong H^*)\subset V^*.$$
The dualization between $V$ and $V^*$ is denoted by $_{V^*}\langle\cdot,\cdot\rangle_V$, then it is clear that $_{V^*}\langle\cdot,\cdot\rangle_V|_{H\times V}=\langle\cdot,\cdot\rangle_H$. $L_2(U,H)$ denotes the space of all Hilbert-Schmidt maps from $U$ to $H$.

 $\mathscr{P}(H)$ stands for the space of all probability measures on $H$ with the weak topology. Furthermore, we set
$$\mathscr{P}_2(H):=\Big\{\mu\in\mathscr{P}(H):\mu(\|\cdot\|_{H}^2):=\int_H\|\xi\|_H^2\mu(d\xi)<\infty\Big\}.$$
Then $\mathscr{P}_2(H)$ is a Polish space under the $L^2$-Wasserstein distance
$$\mathbb{W}_{2,H}(\mu,\nu):=\inf_{\pi\in\mathscr{C}(\mu,\nu)}\Big(\int_{H\times H}\|\xi-\eta\|_H^2\pi(d\xi,d\eta)\Big)^{\frac{1}{2}},~\mu,\nu\in\mathscr{P}_2(H),$$
here $\mathscr{C}(\mu,\nu)$ stands for the set of all couplings for the measures $\mu$ and $\nu$, i.e., $\pi\in\mathscr{C}(\mu,\nu)$ is a probability measure on $H\times H$ such that $\pi(\cdot\times H)=\mu$ and $\pi(H\times \cdot)=\nu$.

Let $T>0$ be fixed. For some measurable maps
$$
A:[0,T]\times V\times\mathscr{P}(H)\rightarrow V^*,~~B:[0,T]\times V\times\mathscr{P}(H)\rightarrow L_2(U,H),
$$
we consider the following type of McKean-Vlasov stochastic evolution equation on $H$,
\begin{equation}\label{e1}
dX_t=A(t,X_t,\mathscr{L}_{X_t})dt+B(t,X_t,\mathscr{L}_{X_t})dW_t,
\end{equation}
where  $\{W_t\}_{t\in [0,T]}$ is an $U$-valued cylindrical Wiener process defined on a complete filtered probability space $\left(\Omega,\mathscr{F},\mathscr{F}_{t\geq0},\mathbb{P}\right)$ (i.e. the path
of $W$ take values in  $C([0,T];U_1)$, where $U_1$ is another
Hilbert space in which the embedding $U\subset U_1$ is
Hilbert--Schmidt).

In this paper, we impose that $A$ and $B$ satisfy the following assumptions.

\begin{hypothesis}\label{h1}
There are some constants $\alpha>1$, $C,\theta,\gamma>0$ such that the following conditions hold.
\begin{enumerate}
\item [$({\mathbf{H}}{\mathbf{1}})$]$(\text{Demicontinuity})$ For all  $t\in[0,T]$, $v\in V$, the map
\begin{eqnarray*}
V\times\mathscr{P}_2(H)\ni(u,\mu)\mapsto_{V^*}\langle A(t,u,\mu),v\rangle_V
\end{eqnarray*}
is continuous.
\item [$({\mathbf{H}}{\mathbf{2}})$]$(\text{Coercivity})$ For all $u\in V$ and $\mu\in\mathscr{P}_2(H)$,
\begin{eqnarray*}
2_{V^*}\langle A(\cdot,u,\mu),u\rangle_V+\|B(\cdot,u,\mu)\|_{L_2(U,H)}^2\leq C\|u\|_H^2+C\mu(\|\cdot\|_H^2)-\theta\|u\|_V^\alpha+C~\text{on}~[0,T].
\end{eqnarray*}
\item [$({\mathbf{H}}{\mathbf{3}})$]$(\text{Monotonicity and Lipschitz})$ For all $u,v\in V$ and $\mu,\nu\in\mathscr{P}_2(H)$,
\begin{eqnarray*}
2_{V^*}\langle A(\cdot,u,\mu)-A(\cdot,v,\nu),u-v\rangle_V \leq C\big(\|u-v\|_H^2+\mathbb{W}_{2,H}(\mu,\nu)^2\big)~\text{on}~[0,T]
\end{eqnarray*}
and
\begin{eqnarray*}
\|B(\cdot,u,\mu)-B(\cdot,v,\nu)\|_{L_2(U,H)}^2\leq C\big(\|u-v\|_H^2+\mathbb{W}_{2,H}(\mu,\nu)^2\big)~\text{on}~[0,T].
\end{eqnarray*}
\item [$({\mathbf{H}}{\mathbf{4}})$]$(\text{Growth})$ For all $u\in V$ and $\mu\in\mathscr{P}_2(H)$,
\begin{eqnarray*}
\|A(\cdot,u,\mu)\|_{V^*}^{\frac{\alpha}{\alpha-1}}\leq C\big(1+\|u\|_V^{\alpha}+\mu(\|\cdot\|_H^2)\big)~\text{on}~[0,T].
\end{eqnarray*}

\item [$({\mathbf{H}}{\mathbf{5}})$]$(\text{Time H\"{o}lder continuity})$ For all $u\in V$, $\mu\in\mathscr{P}_2(H)$ and $t,s\in[0,T]$,
\begin{eqnarray*}
\|B(t,u,\mu)-B(s,u,\mu)\|_{L_2(U,H)}\leq C\Big(1+\|u\|_H+\sqrt{\mu(\|\cdot\|_H^2)}\Big)|t-s|^\gamma.
\end{eqnarray*}

\end{enumerate}
\end{hypothesis}
\begin{Rem}
The conditions $({\mathbf{H}}{\mathbf{1}})$-$({\mathbf{H}}{\mathbf{4}})$ will be used to guarantee the existence and uniqueness of solution to Eq.~(\ref{e1}). In order to investigate the LDP, we further assume the time H\"{o}lder continuity $({\mathbf{H}}{\mathbf{5}})$ of diffusion coefficient $B$. Moreover, it should be mentioned that if $B$ is time homogeneous, then $({\mathbf{H}}{\mathbf{5}})$ is automatically fulfilled.
\end{Rem}
\
\begin{definition}\label{d1}
We call a continuous $H$-valued $(\mathscr{F}_t)_{t\geq 0}$-adapted process $\{X_t\}_{t\in[0,T]}$ is a solution of Eq.~(\ref{e1}), if for its $dt\times \mathbb{P}$-equivalent class $\hat{X}$
\begin{eqnarray*}
\hat{X}\in L^\alpha\big([0,T]\times\Omega,dt\times\mathbb{P};V\big)\cap L^2\big([0,T]\times\Omega,dt\times\mathbb{P};H\big),
\end{eqnarray*}
where $\alpha$ is the same as defined in $({\mathbf{H}}{\mathbf{2}})$ and $\mathbb{P}$-a.s.
\begin{eqnarray*}
X_t=X_0+\int_0^t A(s,\bar{X}_s,\mathscr{L}_{\bar{X}_s})ds+\int_0^t B(s,\bar{X}_s,\mathscr{L}_{\bar{X}_s})dW_s,~t\in[0,T],
\end{eqnarray*}
here $\bar{X}$ is an $V$-valued progressively measurable $dt\times\mathbb{P}$-version of $\hat{X}$.
\end{definition}

We recall the following existence and uniqueness result to Eq.~(\ref{e1}), which has been established in \cite[Theorem 2.1]{HL}.
\begin{proposition}\label{p1}
Assume $({\mathbf{H}}{\mathbf{1}})$-$({\mathbf{H}}{\mathbf{4}})$, for any initial value $X_0\in L^2(\Omega; H)$, Eq.~(\ref{e1}) admits a unique solution $\{X_t\}_{t\in[0,T]}$ in the sense of Definition \ref{d1} and satisfies that
$$E\Big[\sup_{t\in[0,T]}\|X_t\|_H^2\Big]<\infty.$$
\end{proposition}

\subsection{Large deviation principle}
In this article, we consider the following MVSPDE with small Gaussian noise,
\begin{equation}\label{e2}
dX_t^\varepsilon=A(t,X_t^\varepsilon,\mathscr{L}_{X_t^\varepsilon})dt+\sqrt{\varepsilon}B(t,X_t^\varepsilon,\mathscr{L}_{X_t^\varepsilon})dW_t,~X_0^\varepsilon=x\in H,
\end{equation}
 where $\varepsilon>0$. According to Proposition \ref{p1}, there exists a unique strong solution $\{X^\varepsilon\}$ of Eq.~(\ref{e2}) taking values in $C([0,T];H)\cap L^{\alpha}([0,T];V)$.

We first recall some definitions and classical results with respect to the theory of
LDP. Let $\{X^\varepsilon\}$ denote a family of
random elements defined on a probability space
$(\Omega,\mathscr{F},\mathbb{P})$ taking values in a Polish
space $E$. The LDP mainly
describes the exponential decay of remote tails of a family of probability distributions, and the rate of
such exponential decay is characterized by the ``rate function".

\begin{definition}(Rate function) A function $I: E\to [0,+\infty)$ is called
a rate function if $I$ is lower semicontinuous. Moreover, a rate function $I$
is called a {\it good rate function} if  the level set $\{x\in E: I(x)\le
K\}$ is compact for each constant $K<\infty$.
\end{definition}

\begin{definition}(Large deviation principle) The random variable family
 $\{X^\varepsilon\}$ is said to satisfy
 the LDP on $E$ with rate function
 $I$ if  the following lower and upper bound conditions hold,

(i) (Lower bound) for any open set $G\subset E$,
$$\liminf_{\varepsilon\to 0}
   \varepsilon \log \mathbb{P}(X^{\varepsilon}\in G)\geq -\inf_{x\in G}I(x).$$

(ii) (Upper bound) for any closed set $F\subset E$,
$$ \limsup_{\varepsilon\to 0}
   \varepsilon \log \mathbb{P}(X^{\varepsilon}\in F)\leq
  -\inf_{x\in F} I(x).
$$
\end{definition}

Now we recall the equivalence between
the LDP and the Laplace principle that is defined by the following (cf. \cite{C1,DE,DZ}).

\begin{definition}\label{d2}(Laplace principle) $\{X^\varepsilon\}$ is
said to satisfy the Laplace principle on $E$ with a rate function
$I$ if for each bounded continuous real-valued function $h$ defined
on $E$, we have
$$\lim_{\varepsilon\to 0}\varepsilon \log \mathbb{E}\left\lbrace
 \exp\left[-\frac{1}{\varepsilon} h(X^{\varepsilon})\right]\right\rbrace
= -\inf_{x\in E}\left\{h(x)+I(x)\right\}.$$
\end{definition}

It is well-known that if $E$ is a Polish space and $I$ is a good rate function, then the
LDP and Laplace principle are equivalent according to Varadhan's Lemma \cite{V1} and Bryc's converse \cite{DZ}.

In this paper, we will choose the path space $E:=C([0,T];H)$. Let
$$\mathcal{A}=\left\lbrace \phi: \phi\  \text{is  $U$-valued
 $\mathscr{F}_t$-predictable process and}\
  \int_0^T\|\phi_s(\omega)\|^2_U\d s<\infty \  \mathbb{P}\text{-}a.s.\right\rbrace, $$
  and
$$S_M=\left\lbrace \phi\in L^2([0,T], U):
\int_0^T\|\phi_s\|^2_{U}  ds\leq M
 \right\rbrace.$$
It is clear that $S_M$ endowed with the weak topology is a Polish space (here and in the sequel, we
 always consider the weak topology on $S_M$). We also define
$$\mathcal{A}_M=\left\{\phi\in\mathcal{A}: \phi_{\cdot}(\omega)\in S_M, ~\mathbb{P}\text{-}a.s.\right\}.$$

For any fixed $\bar{\mu}_{\cdot}\in C([0,T];\mathscr{P}(H))$, we define the following reference SPDE on $H$,
\begin{equation}\label{e3}
\left\{ \begin{aligned}
&dY_t=A(t,Y_t,\bar{\mu}_t)dt+B(t,Y_t,\bar{\mu}_t)dW_t,~0\leq t\leq T,\\
&Y_0=y\in H.
\end{aligned} \right.
\end{equation}

The definition of solutions to Eq.~(\ref{e3}) is given as follows.
\begin{definition}\label{d3}
For any fixed $\bar{\mu}_{\cdot}\in C([0,T];\mathscr{P}(H))$, we call a continuous $H$-valued $(\mathscr{F}_t)_{t\geq 0}$-adapted process $\{Y_t\}_{t\in[0,T]}$ is a solution of Eq.~(\ref{e3}), if for its $dt\times \mathbb{P}$-equivalent class $\hat{Y}$
\begin{eqnarray*}
\hat{Y}\in L^\alpha\big([0,T]\times\Omega,dt\times\mathbb{P};V\big)\cap L^2\big([0,T]\times\Omega,dt\times\mathbb{P};H\big),
\end{eqnarray*}
where $\alpha$ is the same as defined in $({\mathbf{H}}{\mathbf{2}})$ and $\mathbb{P}$-a.s.
\begin{eqnarray*}
Y_t=Y_0+\int_0^t A(s,\bar{Y}_s,\bar{\mu}_s)ds+\int_0^t B(s,\bar{Y}_s,\bar{\mu}_s)dW_s,~t\in[0,T],
\end{eqnarray*}
here $\bar{Y}$ is an $V$-valued progressively measurable $dt\times\mathbb{P}$-version of $\hat{Y}$.
\end{definition}

Following from the classical well-posedness result (cf.~\cite[Theorem 4.2.5]{LR1}), Eq.~(\ref{e3}) admits the existence of unique strong solution due to the assumptions  (\textbf{H1})-(\textbf{H4}), hence the pathwise uniqueness holds.

Let us state the following important lemmas, which give a variational representation formula for measurable functionals of Wiener process with respect to the distribution dependent case and play an essential role in proving the LDP.
\begin{lemma} (\cite[Theorem 3.6]{LSZZ})\label{l1}
For any $\bar{\mu}_{\cdot}\in C([0,T];\mathscr{P}(H))$, suppose that the pathwise uniqueness for Eq.~(\ref{e3}) holds. Then there exists a map $\mathcal{G}_{\bar{\mu}}:C([0,T];U_1)\to E$ such that the solution is given by
$$Y_{\cdot}=\mathcal{G}_{\bar{\mu}}(W_{\cdot}).$$
Moreover, for any $\phi\in \mathcal{A}_M$, define
$$Y^{\phi}_{\cdot}:=\mathcal{G}_{\bar{\mu}}\Big(W_{\cdot}+\int_0^{\cdot}\phi_sds\Big),$$
then

(i) $Y^{\phi}=\{Y^{\phi}_t,~t\in[0,T]\}$ is an $\mathscr{F}_t$-adapted process with paths in $E$.

(ii) process $Y^{\phi}$ fulfills the following equation
\begin{eqnarray*}
Y^{\phi}_t=y+\int_0^tA(s,Y^{\phi}_s,\bar{\mu}_s)ds+\int_0^t B(s,Y^{\phi}_s,\bar{\mu}_s)dW_s+\int_0^t B(s,Y^{\phi}_s,\bar{\mu}_s)\phi_sds,~t\in[0,T],~\mathbb{P}\text{-a.s.}.
\end{eqnarray*}
\end{lemma}

The above lemma immediately leads to the following result.
\begin{lemma} \label{l2}
Assume that $X$ is a solution to Eq.~(\ref{e1}) with initial value $X_0=x$ and the pathwise uniqueness is satisfied for Eq.~(\ref{e3}) with $y=x$ and $\bar{\mu}_t=\mathscr{L}_{X_t}$,~$t\in[0,T]$. Then $X_\cdot=\mathcal{G}_{\mathscr{L}_{X}}(W_\cdot)$, where $\mathcal{G}_{\mathscr{L}_{X}}$ is defined in Lemma \ref{l1} with $\bar{\mu}=\mathscr{L}_{X}$.
Moreover, for any $\phi\in \mathcal{A}_M$, define
$$X^{\phi}_\cdot:=\mathcal{G}_{\mathscr{L}_{X}}\Big(W_{\cdot}+\int_0^{\cdot}\phi_sds\Big),$$
then

(i) $X^{\phi}=\{X^{\phi}_t,~t\in[0,T]\}$ is an $\mathscr{F}_t$-adapted process with paths in $E$.

(ii) process $X^{\phi}$ fulfills the following equation
\begin{eqnarray*}
X^{\phi}_t=x+\int_0^tA(s,X^{\phi}_s,\mathscr{L}_{X_s})ds+\int_0^t B(s,X^{\phi}_s,\mathscr{L}_{X_s})dW_s+\int_0^t B(s,X^{\phi}_s,\mathscr{L}_{X_s})\phi_sds,~\mathbb{P}\text{-a.s.}.
\end{eqnarray*}
\end{lemma}

Hence, for the Eq.~(\ref{e2}) there exists a map
$$\mathcal{G}^\varepsilon:=\mathcal{G}_{\mathscr{L}_{X^\varepsilon}},$$
where the right-hand side is from Lemma \ref{l2}, such that $X^\varepsilon_\cdot=\mathcal{G}^\varepsilon(\sqrt{\varepsilon}W_{\cdot})$. Furthermore, for each $\phi^\varepsilon\in \mathcal{A}_M$, define
$$X^{\varepsilon,\phi^\varepsilon}_\cdot:=\mathcal{G}^\varepsilon\Big(\sqrt{\varepsilon}W_{\cdot}+\int_0^{\cdot}\phi^\varepsilon_sds\Big),$$
then process $X^{\varepsilon,\phi^\varepsilon}$ solves the following controlled equation $\mathbb{P}$-a.s.
\begin{eqnarray}\label{e5}
X^{\varepsilon,\phi^\varepsilon}_t=\!\!\!\!\!\!\!\!&&x+\int_0^tA(s,X^{\varepsilon,\phi^\varepsilon}_s,\mathscr{L}_{X^{\varepsilon}_s})ds+\int_0^t B(s,X^{\varepsilon,\phi^\varepsilon}_s,\mathscr{L}_{X^{\varepsilon}_s})\phi^\varepsilon_sds
\nonumber\\
\!\!\!\!\!\!\!\!&&+\sqrt{\varepsilon}\int_0^t B(s,X^{\varepsilon,\phi^\varepsilon}_s,\mathscr{L}_{X^{\varepsilon}_s})dW_s.
\end{eqnarray}

We now formulate the following sufficient condition established in \cite{LSZZ} (cf. also \cite{MSZ}) for the Laplace
 principle (equivalently, the LDP) of $X^\varepsilon$ as
 $\varepsilon\rightarrow0$, which is a modified form of \cite{BD,BDM} and is convenient in some applications.

\vspace{1mm}
\textbf{Condition (A)}: There exists a measurable map $\mathcal{G}^0: C([0,T];
U_1)\rightarrow E$ for which the following two conditions hold:

(i) Let $\{\phi^\varepsilon: \varepsilon>0\}\subset \mathcal{A}_M$ for
any $M<\infty$. For any $\delta>0$,
$$\lim_{\varepsilon\to 0}\mathbb{P}\Big(d\Big(\mathcal{G}^\varepsilon\big(\sqrt{\varepsilon}W_{\cdot}+\int_0^{\cdot}\phi^\varepsilon_sds\big),\mathcal{G}^0\big(\int_0^{\cdot}\phi^\varepsilon_sds\big)\Big)>\delta\Big)=0, $$
where $d(\cdot,\cdot)$ denotes the metric on the path space $E$.

(ii) Let $\{\phi^n: n\in\mathbb{N}\}\subset S_M$ for any $M<\infty$ such that $\phi^n$ converges to element $\phi$ in $S_M$ as $n\to\infty$, then
$\mathcal{G}^0\big(\int_0^{\cdot}\phi^n_sds\big)$ converges to $\mathcal{G}^0\big(\int_0^{\cdot}\phi_sds\big)$ in the space $E$.

\begin{lemma}\label{l3}\cite[Theorem 4.4]{LSZZ}  If
$X^\varepsilon_\cdot=\mathcal{G}^\varepsilon(\sqrt{\varepsilon}W_{\cdot})$ and \textbf{Condition (A)}
holds, then $\{X^\varepsilon\}$ satisfies the Laplace
principle (hence LDP) on $E$ with the good
rate function $I$ given by
\begin{equation}\label{rf}
I(f)=\inf_{\left\{\phi\in L^2([0,T]; U):\  f=\mathcal{G}^0(\int_0^\cdot
\phi_sds)\right\}}\left\lbrace\frac{1}{2}
\int_0^T\|\phi_s\|_U^2ds \right\rbrace,
\end{equation}
where infimum over an empty set is taken as $+\infty$.
\end{lemma}

It is well-known that $\big(C([0,T]; H),d(\cdot,\cdot)\big)$ is a Polish space with respect to the metric
\begin{equation*}\label{rho}
d(f,g):=\sup_{t\in[0,T]}\|f_t-g_t\|_{H}.
\end{equation*}

Using Proposition \ref{p1}, it is easy to derive the following proposition.
\begin{proposition}\label{p2}
Suppose that $({\mathbf{H}}{\mathbf{1}})$-$({\mathbf{H}}{\mathbf{4}})$ hold, there exists a unique function $\{X^0_t\}_{t\in[0,T]}$ such that

(i) $X^0\in C([0,T]; H)$,

(ii) $X^0$ solves the following deterministic equation
\begin{equation}\label{e6}
X^0_t=x+\int_0^tA(s,X^0_s,\mathscr{L}_{X^0_s})ds,~t\in[0,T].
\end{equation}
\end{proposition}

We would like to mention that $X^0$ is a deterministic path and its distribution $\mathscr{L}_{X^0_t}=\delta_{X^0_t}$, here $\delta_{X^0_t}$ denotes the Dirac measure of $X^0_t$. Throughout this paper, we always use $X^0$ to denote the unique solution of Eq.~(\ref{e6}).

Now we are in the position to define the following skeleton equation
\begin{equation}\label{e4}
\frac{d\bar{X}^{\phi}_t}{dt}=A(t,\bar{X}^{\phi}_t,\mathscr{L}_{X^0_t})+B(t,\bar{X}^{\phi}_t,\mathscr{L}_{X^0_t})\phi_t,~~\bar{X}^{\phi}_0=x,
\end{equation}
where $\phi\in L^2([0,T];U)$ and $X^0$ is defined in Eq.~(\ref{e6}).
\begin{Rem}
 Note that $\mathscr{L}_{X^0_t}$ shows up in the skeleton equation \eref{e4} and is used to define the large deviation rate function instead of  $\mathscr{L}_{\bar{X}_t^\phi}$. Intuitively, as the parameter $\varepsilon$ tends to $0$ in (\ref{e2}), the noise term vanishes and we arrive the PDE (\ref{e6}). Therefore we can see that the distribution of $X^\varepsilon_t$ tends to the Dirac measure of deterministic trajectory $X^0_t$ as $\varepsilon\to 0$.

Since the measure $\mathscr{L}_{X_t^\varepsilon}$ is not random, so that the convergence of these
measures is independent of the occurrence of a rare event for the random variable $X_t^\varepsilon$. In fact, if one formulate the correct form of equation \eref{e5} for the controlled process $X_t^{\varepsilon, \phi^\varepsilon}$, then let $\varepsilon$ go to $0$, it's natural to derive the above skeleton equation \eref{e4}.
Moreover, in the work of \cite{DST}, the authors also state that the Dirac measure $\mathscr{L}_{X^0_t}$ is a good approximation of $\mathscr{L}_{X_t^\varepsilon}$, where they first replace the distribution $\mathscr{L}_{X_t^\varepsilon}$ in Eq.~(\ref{e2}) by $\mathscr{L}_{X^0_t}$, and derive the skeleton equation \eref{e4}.
Then they use some time discretization, approximation and exponential equivalence arguments to show $X_t^\varepsilon$ in Eq.~(\ref{e2}) satisfies the LDP.
Therefore using $\mathscr{L}_{X^0_t}$ in the skeleton equation to define the rate function instead of $\mathscr{L}_{\bar{X}_t^\phi}$ is reasonable.
\end{Rem}

The existence and uniqueness of solutions to Eq.~(\ref{e4}) for any $\phi\in L^2([0,T];U)$ will be proved in the next section (see Lemma \ref{l4} below).  Furthermore, Lemmas \ref{l3} and \ref{l4} allow us to define
the map $\mathcal{G}^0: C([0,T]; U_1)\rightarrow C([0,T]; H)$ by
\begin{equation}\label{g1}
\mathcal{G}^0\Big(\int_0^{\cdot}\phi_sds\Big):=\bar{X}^{\phi}_{\cdot}.
\end{equation}

Now we can state the main LDP result of this work.
\begin{theorem}\label{t1}
Assume that $({\mathbf{H}}{\mathbf{1}})$-$({\mathbf{H}}{\mathbf{5}})$ and $\int_0^T\|B(s,0,\delta_0)\|_{L_2(U,H)}^2ds<\infty$ hold. Then as $\varepsilon\to 0$, $\{X^\varepsilon:\varepsilon>0\}$
satisfies the LDP on $C([0,T]; H)$ with the
good rate function $I$ given by $(\ref{rf})$.
\end{theorem}

Throughout the present paper, we denote by $C_{p_1,p_2,\cdots}$ some positive constant whose value may change from
line to line, and depends only on the variables $p_1,p_2,\cdots$.

\section{Proof of main result}
\setcounter{equation}{0}
 \setcounter{definition}{0}
We first show the existence and uniqueness of solutions with some priori estimates to the skeleton equation (\ref{e4}), and then we verify \textbf{Condition (A)} which gives the LDP for  $\{X^\varepsilon:\varepsilon>0\}$ on  $C([0,T]; H)$. In particular, in order to avoid the compactness assumption in the Gelfand triple, some time discretization techniques will also be employed.

\begin{lemma}\label{l4}
Suppose that $({\mathbf{H}}{\mathbf{1}})$-$({\mathbf{H}}{\mathbf{4}})$ hold. For every $x\in H$ and $\phi\in L^2([0,T];U)$, there exists a unique solution $\{\bar{X}^{\phi}_t\}_{t\in[0, T]}$ to Eq.~(\ref{e4}) fulfilling
\begin{equation}\label{a1}
\sup_{\phi\in S_M}\Big\{\sup_{t\in[0,T]}\|\bar{X}^{\phi}_t\|_{H}^2+\theta\int_0^T\|\bar{X}^{\phi}_t\|_{V}^{\alpha}dt\Big\}\leq C_{T,M}(1+\|x\|_{H}^2+\sup_{t\in[0,T]}\|X^0_t\|_H^2),
\end{equation}
where $C_{T,M}$ is a positive constant.
\end{lemma}
\begin{proof}
To verify the well-posedness result of Eq.~(\ref{e4}), we first consider $\phi\in L^{\infty}([0,T];U)$ and let
$$\widetilde{A}_{\mu}(t,u):= A(t,u,\mu)+B(t,u,\mu)\phi_t.$$
In terms of ({\textbf{H}}{\textbf{1}})-({\textbf{H}}{\textbf{4}}), it is enough to check that the conditions presented in \cite[Theorem 4.2.5]{LR1} hold for $\widetilde{A}_{\mu}$. Therefore Eq.~(\ref{e4}) has a unique strong solution $\bar{X}^{\phi}$ fulfilling
$$\sup_{t\in[0,T]}\|\bar{X}^{\phi}_t\|_{H}^2+\int_0^T\|\bar{X}^{\phi}_t\|_{V}^{\alpha}dt<\infty,~\phi\in L^{\infty}([0,T];U).$$
For any element $\phi\in L^2([0,T];U)$, it is easy to find a sequence
$\phi^n\in L^\infty([0,T];U)$ such that $\phi^n$ converges strongly to $\phi$ in $L^2([0,T];U)$ as $n\to\infty$. Let us denote by $\bar{X}^{\phi^n}$ the unique solution to Eq.~(\ref{e4}) with $\phi^n\in L^{\infty}([0,T];U)$. For simplicity of notations, we denote $\mu^0_t:=\mathscr{L}_{X^0_t}$. Using ({\textbf{H}}{\textbf{3}}),  for any $n,m\in\mathbb{N}$, there is a  constant $C>0$ such that
\begin{eqnarray}\label{1}
\!\!\!\!\!\!\!\!&&\frac{d}{dt}\|\bar{X}^{\phi^n}_t-\bar{X}^{\phi^m}_t\|_{H}^2
\nonumber\\
~=\!\!\!\!\!\!\!\!&&2{}_{V^*}\langle A(t,\bar{X}^{\phi^n}_t,\mu^0_t)-A(t,\bar{X}^{\phi^m}_t,\mu^0_t),\bar{X}^{\phi^n}_t-\bar{X}^{\phi^m}_t\rangle_{V}
\nonumber\\
\!\!\!\!\!\!\!\!&&+2\langle B(t,\bar{X}^{\phi^n}_t,\mu^0_t)\phi^n_t-B(t,\bar{X}^{\phi^m}_t,\mu^0_t)\phi^m_t,\bar{X}^{\phi^n}_t-\bar{X}^{\phi^m}_t\rangle_{H}
\nonumber\\
~\leq\!\!\!\!\!\!\!\!&&2{}_{V^*}\langle A(t,\bar{X}^{\phi^n}_t,\mu^0_t)-A(t,\bar{X}^{\phi^m}_t,\mu^0_t),\bar{X}^{\phi^n}_t-\bar{X}^{\phi^m}_t\rangle_{V}
\nonumber\\
\!\!\!\!\!\!\!\!&&
+\|B(t,\bar{X}^{\phi^n}_t,\mu^0_t)-B(t,\bar{X}^{\phi^m}_t,\mu^0_t)\|_{L_2(U,H)}^2+\|\phi^n_t\|_U^2\|\bar{X}^{\phi^n}_t-\bar{X}^{\phi^m}_t\|_{H}^2
\nonumber\\
\!\!\!\!\!\!\!\!&&+2\langle B(t,\bar{X}^{\phi^m}_t,\mu^0_t)(\phi^n_t-\phi^m_t),\bar{X}^{\phi^n}_t-\bar{X}^{\phi^m}_t\rangle_{H}
\nonumber\\
~\leq\!\!\!\!\!\!\!\!&&\|\phi^n_t-\phi^m_t\|_U^2+(C+\|\phi^n_t\|_U^2+\|B(t,\bar{X}^{\phi^m}_t,\mu^0_t)\|_{L_2(U,H)}^2)\|\bar{X}^{\phi^n}_t-\bar{X}^{\phi^m}_t\|_{H}^2.
\end{eqnarray}
 Applying Gronwall' inequality and ({\textbf{H}}{\textbf{3}}) gives that
\begin{eqnarray}\label{4}
\!\!\!\!\!\!\!\!&&\|\bar{X}^{\phi^n}_t-\bar{X}^{\phi^m}_t\|_{H}^2
\nonumber\\
~\leq\!\!\!\!\!\!\!\!&&\exp\Big\{\int_0^T\Big(C+\|\phi^n_t\|_U^2+\|B(t,\bar{X}^{\phi^m}_t,\mu^0_t)\|_{L_2(U,H)}^2\Big)dt\Big\}\int_0^T\|\phi^n_t-\phi^m_t\|_U^2dt
\nonumber\\
~\leq\!\!\!\!\!\!\!\!&&C\exp\Big\{\int_0^T\Big(1+\|\phi^n_t\|_U^2+\|\bar{X}^{\phi^m}_t\|_H^2+\mu^0_t(\|\cdot\|_H^2)\Big)dt\Big\}\int_0^T\|\phi^n_t-\phi^m_t\|_U^2dt.
\end{eqnarray}
Repeating the similar arguments as in (\ref{1}), one can get
\begin{eqnarray*}
\frac{d}{dt}\|\bar{X}^{\phi^m}_t\|_{H}^2=\!\!\!\!\!\!\!\!&&2{}_{V^*}\langle A(t,\bar{X}^{\phi^m}_t,\mu^0_t),\bar{X}^{\phi^m}_t\rangle_{V}+2\langle B(t,\bar{X}^{\phi^m}_t,\mu^0_t)\phi^m_t,\bar{X}^{\phi^m}_t\rangle_{H}
\nonumber\\
\leq\!\!\!\!\!\!\!\!&&-\theta\|\bar{X}^{\phi^m}_t\|_{V}^{\alpha}+C(1+\|\phi^m_t\|_U^2)\|\bar{X}^{\phi^m}_t\|_{H}^2+C(1+\mu^0_t(\|\cdot\|_H^2)).
\end{eqnarray*}
Note that $X^0\in C([0,T];H)$, and $\mu^0_t(\|\cdot\|_H^2)=\|X_t^0\|_H^2$. Then, for any $\phi^m\in S_M$, Gronwall' inequality yields that
\begin{eqnarray}\label{2}
\!\!\!\!\!\!\!\!&&\sup_{t\in[0,T]}\|\bar{X}^{\phi^m}_t\|_{H}^2+\theta\int_0^T\|\bar{X}^{\phi^m}_t\|_{V}^{\alpha}dt
\nonumber\\
\leq\!\!\!\!\!\!\!\!&&C_T\exp\Big\{\int_0^T\Big(1+\|\phi^m_t\|_U^2\Big)dt\Big\}(1+\|x\|_{H}^2+\sup_{t\in[0,T]}\|X^0_t\|_H^2)
\nonumber\\
\leq\!\!\!\!\!\!\!\!&&C_{T,M}(1+\|x\|_{H}^2+\sup_{t\in[0,T]}\|X^0_t\|_H^2),
\end{eqnarray}
where the constant $C_{T,M}$ only depends on $T,M$.

Combining ({\textbf{H}}{\textbf{3}}) with (\ref{2}) leads to
\begin{eqnarray}\label{3}
\int_0^T\|B(t,\bar{X}^{\phi^m}_t,\mu^0_t)\|_{L_2(U,H)}^2dt\leq\!\!\!\!\!\!\!\!&& C\int_0^T\big(1+\|\bar{X}^{\phi^m}_t\|_{H}^2+\mu^0_t(\|\cdot\|_H^2)\big)dt
\nonumber\\
\leq\!\!\!\!\!\!\!\!&&C_{T,M}(1+\|x\|_{H}^2+\sup_{t\in[0,T]}\|X^0_t\|_H^2).
\end{eqnarray}

Consequently, substituting (\ref{2}) and (\ref{3}) into (\ref{4}) and then letting $n,m\to\infty$, it is easy to see that $\{\bar{X}^{\phi^n}\}_{n\geq1}$ is a Cauchy sequence in $C([0,T];H)$ and we denote the limit by $\bar{X}^{\phi}$. Making use of the standard monotonicity argument (cf. e.g.~\cite[Theorem 30.A]{Z}), one can get that $\bar{X}^{\phi}$ is a solution to Eq.~(\ref{e4}) corresponding to $\phi$. The pathwise uniqueness follows from the condition ({\textbf{H}}{\textbf{3}}) and Gronwall's inequality. In particular, the energy estimate (\ref{a1}) is a direct consequence of (\ref{2}).

The proof of Lemma \ref{l4} is completed. \hspace{\fill}$\Box$
\end{proof}

In the following, we devote to presenting that $\{X^\varepsilon:\varepsilon>0\}$ satisfies the LDP on $C([0,T]; H)$ by proving \textbf{Condition (A)} (i) and (ii). The proof of \textbf{Condition (A)} (i) will be given in Theorem \ref{t2} and \textbf{Condition (A)} (ii) will be established in Theorem \ref{t3}.

For any $\phi^\varepsilon\in \mathcal{A}_M$, $\varepsilon>0$, we recall that
$$X^{\varepsilon,\phi^\varepsilon}:=\mathcal{G}^\varepsilon\Big(\sqrt{\varepsilon}W_{\cdot}+\int_0^{\cdot}\phi^\varepsilon_sds\Big)$$
solves the following stochastic control equation
\begin{equation}\label{e7}
\left\{ \begin{aligned}
X^{\varepsilon,\phi^\varepsilon}_t=&~A(t,X^{\varepsilon,\phi^\varepsilon}_t,\mathscr{L}_{X^{\varepsilon}_t})dt+ B(t,X^{\varepsilon,\phi^\varepsilon}_t,\mathscr{L}_{X^{\varepsilon}_t})\phi^\varepsilon_sdt,\\
&+\sqrt{\varepsilon}B(t,X^{\varepsilon,\phi^\varepsilon}_t,\mathscr{L}_{X^{\varepsilon}_t})dW_t,\\
X^{\varepsilon,\phi^\varepsilon}_0=&~x\in H,
\end{aligned} \right.
\end{equation}
where $X^{\varepsilon}$ is the unique solution defined in Eq.~(\ref{e2}).

Before we give the proof of \textbf{Condition (A)} (i), we first introduce the following lemma which characterizes the difference between $X^\varepsilon$ and $X^0$.
\begin{lemma}\label{l5}
There is a constant $C_T>0$ such that
$$\mathbb{E}\Big[\sup_{t\in[0,T]}\|X^{\varepsilon}_t-X^0_t\|_H^2\Big]\leq C_T\varepsilon\big(\sup_{t\in[0,T]}\|X^0_t\|_H^2\big).$$
\end{lemma}
\begin{proof}
For simplicity of notations, we denote $Z^\varepsilon_t:=X^{\varepsilon}_t-X^0_t$ which solves the following SPDE
\begin{equation*}\label{e8}
\left\{ \begin{aligned}
&Z^\varepsilon_t=\big[A(t,X^{\varepsilon}_t,\mu^\varepsilon_t)-A(t,X^0_t,\mu^0_t)]dt
+\sqrt{\varepsilon}B(t,X^{\varepsilon}_t,\mu^\varepsilon_t)dW_t,\\
&Z^\varepsilon_0=0,
\end{aligned} \right.
\end{equation*}
here we denote $\mu^\varepsilon_t:=\mathscr{L}_{X^{\varepsilon}_t}$ and $\mu^0_t:=\mathscr{L}_{X^0_t}$. Applying It\^{o}'s formula yields that
\begin{eqnarray}\label{12}
\|Z^\varepsilon_t\|_H^2=\!\!\!\!\!\!\!\!&&2\int_0^t{}_{V^*}\langle A(s,X^{\varepsilon}_s,\mu^\varepsilon_s)-A(s,X^0_s,\mu^0_s),Z^\varepsilon_s\rangle_Vds
\nonumber\\
\!\!\!\!\!\!\!\!&&+2\sqrt{\varepsilon}\int_0^t\langle B(s,X^{\varepsilon}_s,\mu^\varepsilon_s)dW_s,Z^\varepsilon_s\rangle_H+\varepsilon\int_0^t\|B(s,X^{\varepsilon}_s,\mu^\varepsilon_s)\|_{L_2(U,H)}^2ds
\nonumber\\
=:\!\!\!\!\!\!\!\!&&\sum_{i=1}^3 I_i(t).
\end{eqnarray}
In terms of the condition ({\textbf{H}}{\textbf{3}}), it follows that
\begin{eqnarray}\label{9}
I_1(t)\leq C\int_0^T\Big(\|Z^\varepsilon_t\|_H^2+\mathbb{W}_{2,H}(\mu^\varepsilon_t,\mu^0_t)^2\Big)dt.
\end{eqnarray}
The term $I_3(t)$ will be estimated as follows
\begin{eqnarray}\label{10}
I_3(t)\leq\!\!\!\!\!\!\!\!&&2\varepsilon\int_0^T\|B(t,X^{\varepsilon}_t,\mu^\varepsilon_t)-B(t,X^0_t,\mu^0_t)\|_{L_2(U,H)}^2dt
+2\varepsilon\int_0^T\|B(t,X^0_t,\mu^0_t)\|_{L_2(U,H)}^2dt
\nonumber\\
\leq\!\!\!\!\!\!\!\!&&C\varepsilon\int_0^T\Big(\|Z^\varepsilon_t\|_H^2+\mathbb{W}_{2,H}(\mu^\varepsilon_t,\mu^0_t)^2\Big)dt+C\varepsilon\int_0^T\Big(1+\|X^0_t\|_H^2+\mu^0_t(\|\cdot\|_H^2)\Big)dt
\nonumber\\
\leq\!\!\!\!\!\!\!\!&&C\varepsilon\int_0^T\Big(\|Z^\varepsilon_t\|_H^2+\mathbb{W}_{2,H}(\mu^\varepsilon_t,\mu^0_t)^2\Big)dt+C\varepsilon\Big(1+\sup_{t\in[0,T]}\|X^0_t\|_H^2\Big).
\end{eqnarray}
For the term $I_2(t)$, note that $\mathbb{W}_{2,H}(\mu^\varepsilon_t,\mu^0_t)^2\leq\mathbb{E}\|X^\varepsilon_t-X^0_t\|_H^2$, applying Burkholder-Davis-Gundy's inequality leads to
\begin{eqnarray}\label{11}
\!\!\!\!\!\!\!\!&&\mathbb{E}\Big[\sup_{t\in[0,T]}|I_2(t)|\Big]
\nonumber\\
\leq\!\!\!\!\!\!\!\!&&2\sqrt{\varepsilon}\mathbb{E}\Big[\sup_{t\in[0,T]}\Big|\int_0^t\langle \big(B(s,X^{\varepsilon}_s,\mu^\varepsilon_s)-B(s,X^0_s,\mu^0_s)\big)dW_s,Z^\varepsilon_s\rangle_H\Big|\Big]
\nonumber\\
\!\!\!\!\!\!\!\!&&+2\sqrt{\varepsilon}\mathbb{E}\Big[\sup_{t\in[0,T]}\Big|\int_0^t\langle B(s,X^0_s,\mu^0_s)dW_s,Z^\varepsilon_s\rangle_H\Big|\Big]
\nonumber\\
\leq\!\!\!\!\!\!\!\!&&C\sqrt{\varepsilon}\mathbb{E}\Big[\int_0^T\|B(t,X^{\varepsilon}_t,\mu^\varepsilon_t)-B(t,X^0_t,\mu^0_t)\|_{L_2(U,H)}^2\|Z^\varepsilon_t\|_H^2dt\Big]^{\frac{1}{2}}
\nonumber\\
\!\!\!\!\!\!\!\!&&+C\sqrt{\varepsilon}\mathbb{E}\Big[\int_0^T\|B(t,X^0_t,\mu^0_t)\|_{L_2(U,H)}^2\|Z^\varepsilon_t\|_H^2dt\Big]^{\frac{1}{2}}
\nonumber\\
\leq\!\!\!\!\!\!\!\!&&\frac{1}{2}\mathbb{E}\Big[\sup_{t\in[0,T]}\|Z^\varepsilon_t\|_H^2\Big]+C\varepsilon\int_0^T\mathbb{E}\|Z^\varepsilon_t\|_H^2dt+C\varepsilon\Big(1+\sup_{t\in[0,T]}\|X^0_t\|_H^2\Big).
\end{eqnarray}
Thus combining (\ref{9})-(\ref{11}) with (\ref{12}) and using Gronwall's inequality, we have
\begin{eqnarray*}
\mathbb{E}\Big[\sup_{t\in[0,T]}\|Z^\varepsilon_t\|_H^2\Big]\leq C_T\varepsilon\Big(1+\sup_{t\in[0,T]}\|X^0_t\|_H^2\Big),
\end{eqnarray*}
which completes the proof of Lemma \ref{l5}. \hspace{\fill}$\Box$
\end{proof}

We would like to recall the map $\mathcal{G}^0$ defined in (\ref{g1}) and get the following result.
\begin{theorem}\label{t2}
Suppose that $({\mathbf{H}}{\mathbf{1}})$-$({\mathbf{H}}{\mathbf{4}})$ and $\int_0^T\|B(s,0,\delta_0)\|_{L_2(U,H)}^2ds<\infty$ hold. Let $\{\phi^\varepsilon: \varepsilon>0\}\subset \mathcal{A}_M$ for
any $M<\infty$. Then for any $\delta>0$,
$$\lim_{\varepsilon\to 0}\mathbb{P}\Big(d\Big(X^{\varepsilon,\phi^\varepsilon}_\cdot,\mathcal{G}^0\big(\int_0^{\cdot}\phi^\varepsilon_sds\big)\Big)>\delta\Big)=0.$$
\end{theorem}
\begin{proof}
Let us denote $\widetilde{Z}^{\varepsilon}_t:=X^{\varepsilon,\phi^\varepsilon}_t-\bar{X}^{\phi^\varepsilon}_t$, which satisfies the following SPDE
\begin{eqnarray*}
\left\{ \begin{aligned}
d\widetilde{Z}^{\varepsilon}_t=&\big[A(t,X^{\varepsilon,\phi^\varepsilon}_t,\mu^\varepsilon_t)-A(t,\bar{X}^{\phi^\varepsilon}_t,\mu^0_t)\big]dt
+\sqrt{\varepsilon}B(t,X^{\varepsilon,\phi^\varepsilon}_t,\mu^\varepsilon_t)dW_t\\
&+\big[B(t,X^{\varepsilon,\phi^\varepsilon}_t,\mu^\varepsilon_t)\phi^\varepsilon_t-B(t,\bar{X}^{\phi^\varepsilon}_t,\mu^0_t)\phi^\varepsilon_t\big]dt,\\
\widetilde{Z}^{\varepsilon}_0=&0.
\end{aligned}\right.
\end{eqnarray*}
Applying It\^{o}'s formula to $\|\widetilde{Z}^{\varepsilon}_t\|_{H}^2$ gives
\begin{eqnarray}\label{5}
\|\widetilde{Z}^{\varepsilon}_t\|_H^2=\!\!\!\!\!\!\!\!&&2\int_0^t{}_{V^*}\langle A(s,X^{\varepsilon,\phi^\varepsilon}_s,\mu^\varepsilon_s)-A(s,\bar{X}^{\phi^\varepsilon}_s,\mu^0_s),\widetilde{Z}^{\varepsilon}_s\rangle_{V}ds
\nonumber\\
\!\!\!\!\!\!\!\!&&+2\int_0^t\langle\big[B(s,X^{\varepsilon,\phi^\varepsilon}_s,\mu^\varepsilon_s)-B(s,\bar{X}^{\phi^\varepsilon}_s,\mu^0_s)\big]\phi^\varepsilon_s,\widetilde{Z}^{\varepsilon}_s\rangle_{H}ds
\nonumber\\
\!\!\!\!\!\!\!\!&&+\varepsilon\int_0^t
\|B(s,X^{\varepsilon,\phi^\varepsilon}_s,\mu^\varepsilon_s)\|_{L_2(U,H)}^2ds+2\sqrt{\varepsilon}\int_0^t\langle B(s,X^{\varepsilon,\phi^\varepsilon}_s,\mu^\varepsilon_s)dW_s,\widetilde{Z}^{\varepsilon}_s\rangle_H
\nonumber\\
=:\!\!\!\!\!\!\!\!&&\sum_{i=1}^{4}\overline{I}_i(t).
\end{eqnarray}
Below we aim to estimate the terms $\overline{I}_i(t)$, $i=1,2,3,4$, respectively.
\begin{eqnarray}\label{6}
\overline{I}_1(t)+\overline{I}_2(t)\leq\!\!\!\!\!\!\!\!&&2\int_0^t{}_{V^*}\langle A(s,X^{\varepsilon,\phi^\varepsilon}_s,\mu^\varepsilon_s)-A(s,\bar{X}^{\phi^\varepsilon}_s,\mu^0_s),\widetilde{Z}^{\varepsilon}_s\rangle_{V}ds
\nonumber\\
\!\!\!\!\!\!\!\!&&
+\int_0^t\|B(s,X^{\varepsilon,\phi^\varepsilon}_s,\mu^\varepsilon_s)-B(s,\bar{X}^{\phi^\varepsilon}_s,\mu^0_s)\|_{L_2(U,H)}^2ds+\int_0^t\|\phi^\varepsilon_s\|_U^2\|\widetilde{Z}^{\varepsilon}_s\|_H^2ds
\nonumber\\
\leq\!\!\!\!\!\!\!\!&&C\int_0^t\Big(\|\widetilde{Z}^{\varepsilon}_s\|_H^2+\mathbb{W}_{2,H}(\mu^\varepsilon_s,\mu^0_s)^2\Big)ds+\int_0^t\|\phi^\varepsilon_s\|_U^2\|\widetilde{Z}^{\varepsilon}_s\|_H^2ds.
\end{eqnarray}
Combining (\ref{6}) with (\ref{5}), one can get that
\begin{eqnarray*}
\|\widetilde{Z}^{\varepsilon}_t\|_H^2\leq\!\!\!\!\!\!\!\!&&C\int_0^t(1+\|\phi^\varepsilon_s\|_U^2)\|\widetilde{Z}^{\varepsilon}_s\|_H^2ds+C\int_0^t\mathbb{W}_{2,H}(\mu^\varepsilon_s,\mu^0_s)^2ds+\overline{I}_3(t)+\overline{I}_4(t).
\end{eqnarray*}
By Gronwall's inequality,
\begin{eqnarray}\label{13}
\|\widetilde{Z}^{\varepsilon}_t\|_H^2\leq\!\!\!\!\!\!\!\!&&\Big[C\int_0^t\mathbb{W}_{2,H}(\mu^\varepsilon_s,\mu^0_s)^2ds+
\overline{I}_3(t)+\overline{I}_4(t)\Big]
\times \exp\Big\{\int_0^t(1+\|\phi^\varepsilon_s\|_U^2)ds\Big\}.
\end{eqnarray}
The term $I_3(t)$ can be controlled as follows
\begin{eqnarray}\label{7}
I_3(t)\leq\!\!\!\!\!\!\!\!&&C\varepsilon\int_0^t
\|B(s,X^{\varepsilon,\phi^\varepsilon}_s,\mu^\varepsilon_s)-B(s,\bar{X}^{\phi^\varepsilon}_s,\mu^0_s)\|_{L_2(U,H)}^2ds
\nonumber\\
\!\!\!\!\!\!\!\!&&+C\varepsilon\int_0^t
\|B(s,\bar{X}^{\phi^\varepsilon}_s,\mu^0_s)-B(s,0,\delta_0)\|_{L_2(U,H)}^2ds+C\varepsilon\int_0^t\|B(s,0,\delta_0)\|_{L_2(U,H)}^2ds
\nonumber\\
\leq\!\!\!\!\!\!\!\!&&C\varepsilon\int_0^t\Big(\|\widetilde{Z}^{\varepsilon}_s\|_H^2+\mathbb{W}_{2,H}(\mu^\varepsilon_s,\mu^0_s)^2\Big)ds
+C\varepsilon\int_0^t
\Big(\|\bar{X}^{\phi^\varepsilon}_s\|_H^2+\mu^0_s(\|\cdot\|_H^2)\Big)ds+C_T\varepsilon
\nonumber\\
\leq\!\!\!\!\!\!\!\!&&C\varepsilon\int_0^t\Big(\|\widetilde{Z}^{\varepsilon}_s\|_H^2+\mathbb{W}_{2,H}(\mu^\varepsilon_s,\mu^0_s)^2\Big)ds
\nonumber\\
\!\!\!\!\!\!\!\!&&
+C_T\varepsilon\Big(1+\sup_{\phi\in S_M}\{\sup_{t\in[0,T]}\|\bar{X}^{\phi}_t\|_H^2\}+\sup_{t\in[0,T]}\|X^0_t\|_H^2\Big).
\end{eqnarray}
Note that $\mathbb{W}_{2,H}(\mu^\varepsilon_t,\mu^0_t)^2\leq\mathbb{E}\|X^\varepsilon_t-X^0_t\|_H^2$, making use of Burkholder-Davis-Gundy's inequality implies that
\begin{eqnarray}\label{8}
\!\!\!\!\!\!\!\!&&\mathbb{E}\big[\sup_{t\in[0,T]}|\overline{I}_4(t)|\big]
\nonumber\\
\leq\!\!\!\!\!\!\!\!&&C\sqrt{\varepsilon}\mathbb{E}\Big[\sup_{t\in[0,T]}\Big|\int_0^t\langle\big(B(s,X^{\varepsilon,\phi^\varepsilon}_s,\mu^\varepsilon_s)
-B(s,\bar{X}^{\phi^\varepsilon}_s,\mu^0_s)\big)dW_s,\widetilde{Z}^{\varepsilon}_s\rangle_H\Big|\Big]
\nonumber\\
\!\!\!\!\!\!\!\!&&+C\sqrt{\varepsilon}\mathbb{E}\Big[\sup_{t\in[0,T]}\Big|\int_0^t\langle \big(B(s,\bar{X}^{\phi^\varepsilon}_s,\mu^0_s)-B(s,0,\delta_0)\big)dW_s,\widetilde{Z}^{\varepsilon}_s\rangle_H\Big|\Big]
\nonumber\\
\!\!\!\!\!\!\!\!&&+C\sqrt{\varepsilon}\mathbb{E}\Big[\sup_{t\in[0,T]}\Big|\int_0^t\langle B(s,0,\delta_0)dW_s,\widetilde{Z}^{\varepsilon}_s\rangle_H\Big|\Big]
\nonumber\\
\leq\!\!\!\!\!\!\!\!&&C\sqrt{\varepsilon}\mathbb{E}\Big[\int_0^T\|B(s,X^{\varepsilon,\phi^\varepsilon}_s,\mu^\varepsilon_s)
-B(s,\bar{X}^{\phi^\varepsilon}_s,\mu^0_s)\|_{L_2(U,H)}^2\|\widetilde{Z}^{\varepsilon}_s\|_H^2ds\Big]^{\frac{1}{2}}
\nonumber\\
\!\!\!\!\!\!\!\!&&+C\sqrt{\varepsilon}\mathbb{E}\Big[\int_0^T\|B(s,\bar{X}^{\phi^\varepsilon}_s,\mu^0_s)-B(s,0,\delta_0)\|_{L_2(U,H)}^2\|\widetilde{Z}^{\varepsilon}_s\|_H^2ds\Big]^{\frac{1}{2}}
\nonumber\\
\!\!\!\!\!\!\!\!&&+C\sqrt{\varepsilon}\mathbb{E}\Big[\int_0^T\|B(s,0,\delta_0)\|_{L_2(U,H)}^2\|\widetilde{Z}^{\varepsilon}_s\|_H^2ds\Big]^{\frac{1}{2}}
\nonumber\\
\leq\!\!\!\!\!\!\!\!&&\varepsilon_0\mathbb{E}\big[\sup_{t\in[0,T]}\|\widetilde{Z}^{\varepsilon}_t\|_H^2\big]+C\varepsilon\mathbb{E}\int_0^T\Big(\|\widetilde{Z}^{\varepsilon}_t\|_H^2+\mathbb{W}_{2,H}(\mu^\varepsilon_t,\mu^0_t)^2\Big)dt
\nonumber\\
\!\!\!\!\!\!\!\!&&+C\varepsilon\mathbb{E}\int_0^T\Big(1+\|\bar{X}^{\phi^\varepsilon}_t\|_H^2+\mu^0_t(\|\cdot\|_H^2)\Big)dt
\nonumber\\
\leq\!\!\!\!\!\!\!\!&&\varepsilon_0\mathbb{E}\big[\sup_{t\in[0,T]}\|\widetilde{Z}^{\varepsilon}_t\|_H^2\big]+C\mathbb{E}\int_0^T\|\widetilde{Z}^{\varepsilon}_t\|_H^2dt
+C\varepsilon\int_0^T\mathbb{E}\|X^\varepsilon_t-X^0_t\|_H^2dt
\nonumber\\
\!\!\!\!\!\!\!\!&&
+C_T\varepsilon\Big(1+\sup_{\phi\in S_M}\{\sup_{t\in[0,T]}\|\bar{X}^{\phi}_t\|_H^2\}+\sup_{t\in[0,T]}\|X^0_t\|_H^2\Big),
\end{eqnarray}
where $\varepsilon_0>0$ is a small enough constant which will be chosen later.

Then substituting (\ref{7})-(\ref{8}) into (\ref{13}) and using $\phi^\varepsilon\in\mathcal{A}_M$ that
\begin{eqnarray*}
\!\!\!\!\!\!\!\!&&\mathbb{E}\Big[\sup_{t\in[0,T]}\|\widetilde{Z}^{\varepsilon}_t\|_H^2\Big]
\nonumber\\
\leq\!\!\!\!\!\!\!\!&&C_{T,M}\mathbb{E}\Big[\int_0^T\mathbb{W}_{2,H}(\mu^\varepsilon_s,\mu^0_s)^2ds+
\sup_{t\in[0,T]}|\overline{I}_3(t)|+\sup_{t\in[0,T]}|\overline{I}_4(t)|\Big]
\nonumber\\
\leq\!\!\!\!\!\!\!\!&&\frac{1}{2}\mathbb{E}\big[\sup_{t\in[0,T]}\|\widetilde{Z}^{\varepsilon}_t\|_H^2\big]+C_{T,M}\mathbb{E}\int_0^T\|\widetilde{Z}^{\varepsilon}_t\|_H^2dt+C_{T,M}\mathbb{E}\Big[\sup_{t\in[0,T]}\|X^\varepsilon_t-X^0_t\|_H^2\Big]
\nonumber\\
\!\!\!\!\!\!\!\!&&
+C_T\varepsilon\Big(1+\sup_{\phi\in S_M}\{\sup_{t\in[0,T]}\|\bar{X}^{\phi}_t\|_H^2\}+\sup_{t\in[0,T]}\|X^0_t\|_H^2\Big),
\end{eqnarray*}
where in the last step we take $\varepsilon_0=\frac{1}{2C_{T,M}}$.

From Gronwall's inequality and Lemma \ref{l5}, we obtain
\begin{eqnarray*}
\!\!\!\!\!\!\!\!&&\mathbb{E}\Big[\sup_{t\in[0,T]}\|\widetilde{Z}^{\varepsilon}_t\|_H^2\Big]
\nonumber\\
\leq\!\!\!\!\!\!\!\!&&C_{T,M}\mathbb{E}\Big[\sup_{t\in[0,T]}\|X^\varepsilon_t-X^0_t\|_H^2\Big]
+C_{T,M}\varepsilon\Big(1+\sup_{\phi\in S_M}\{\sup_{t\in[0,T]}\|\bar{X}^{\phi}_t\|_H^2\}+\sup_{t\in[0,T]}\|X^0_t\|_H^2\Big)
\nonumber\\
\leq\!\!\!\!\!\!\!\!&&C_{T,M}\varepsilon\Big(1+\|x\|_H^2+\sup_{t\in[0,T]}\|X^0_t\|_H^2\Big),
\end{eqnarray*}
where we used the priori estimate (\ref{a1}) in the last step.

Consequently, it is easy to see that
$$\lim_{\varepsilon\to 0}\mathbb{E}\Big[\sup_{t\in[0,T]}\Big\|X^{\varepsilon,\phi^\varepsilon}_t-\mathcal{G}^0\Big(\int_0^{t}\phi^\varepsilon_sds\Big)\Big\|_H^2\Big]=0,$$
which implies the assertion.  \hspace{\fill}$\Box$
\end{proof}

In order to prove \textbf{Condition (A)} (ii), we need the following crucial lemma that concerns the time
increments of solution to the deterministic skeleton equation (\ref{e4}).
\begin{lemma}\label{l6}
For any $x\in H$ and $\phi\in S_M$, there exists a constant $C>0$ such that
$$\int_0^{T}\|\bar{X}^{\phi}_t-\bar{X}^{\phi}_{t(\delta)}\|_{H}^2dt\leq C_{T,M}\delta(1+\|x\|_{H}^2+\sup_{t\in[0,T]}\|X^0_t\|_H^2),$$
here $\delta>0$ is a small enough constant, $t(\delta):=[\frac{t}{\delta}]\delta$ and $[s]$ is the largest integer smaller than $s$.
\end{lemma}
\begin{proof}
First according to Lemma \ref{l4}, it is obvious that
\begin{eqnarray}\label{14}
\int_0^{T}\|\bar{X}^{\phi}_t-\bar{X}^{\phi}_{t(\delta)}\|_{H}^2dt\leq\!\!\!\!\!\!\!\!&&\int_0^{\delta}\|\bar{X}^{\phi}_t-x\|_{H}^2dt+
\int_{\delta}^{T}\|\bar{X}^{\phi}_t-\bar{X}^{\phi}_{t(\delta)}\|_{H}^2dt
\nonumber\\
\leq\!\!\!\!\!\!\!\!&&2\delta\|x\|_{H}^2+2\delta\sup_{\psi\in S_M}\Big\{\sup_{t\in[0,T]}\|\bar{X}^{\psi}_t\|_H^2\Big\}+2\int_{\delta}^{T}\|\bar{X}^{\phi}_t-\bar{X}^{\phi}_{t-\delta}\|_{H}^2dt
\nonumber\\
\!\!\!\!\!\!\!\!&&+2\int_{\delta}^{T}\|\bar{X}^{\phi}_{t(\delta)}-\bar{X}^{\phi}_{t-\delta}\|_{H}^2dt.
\nonumber\\
\leq\!\!\!\!\!\!\!\!&&C_{T,M}\delta(1+\|x\|_{H}^2+\sup_{t\in[0,T]}\|X^0_t\|_H^2)+2\int_{\delta}^{T}\|\bar{X}^{\phi}_t-\bar{X}^{\phi}_{t-\delta}\|_{H}^2dt
\nonumber\\
\!\!\!\!\!\!\!\!&&+2\int_{\delta}^{T}\|\bar{X}^{\phi}_{t(\delta)}-\bar{X}^{\phi}_{t-\delta}\|_{H}^2dt.
\end{eqnarray}
Let us consider the first integral term on the right side of (\ref{14}). By using chain rule we have
\begin{eqnarray}\label{15}
\|\bar{X}^{\phi}_t-\bar{X}^{\phi}_{t-\delta}\|_{H}^2=\!\!\!\!\!\!\!\!&&2\int_{t-\delta}^{t}{}_{V^*}\langle A(s,\bar{X}^{\phi}_s,\mu^0_s),\bar{X}^{\phi}_s-\bar{X}^{\phi}_{t-\delta}\rangle_Vds
\nonumber\\
\!\!\!\!\!\!\!\!&&+2\int_{t-\delta}^{t}\langle B(s,\bar{X}^{\phi}_s,\mu^0_s)\phi_s,\bar{X}^{\phi}_s-\bar{X}^{\phi}_{t-\delta}\rangle_Hds
\nonumber\\
=:\!\!\!\!\!\!\!\!&&K_1(t)+K_2(t).
\end{eqnarray}
We shall estimate the terms $\int_\delta^{T}K_i(t)dt$, $i=1,2$, respectively.
\begin{eqnarray}\label{16}
\!\!\!\!\!\!\!\!&&\int_\delta^{T}K_1(t)dt
\nonumber\\
\leq\!\!\!\!\!\!\!\!&&2\int_\delta^{T}\int_{t-\delta}^{t}\|A(s,\bar{X}^{\phi}_s,\mu^0_s)\|_{V^*}\|\bar{X}^{\phi}_s-\bar{X}^{\phi}_{t-\delta}\|_Vdsdt
\nonumber\\
\leq\!\!\!\!\!\!\!\!&&2\Big[\int_\delta^{T}\int_{t-\delta}^{t}\|A(s,\bar{X}^{\phi}_s,\mu^0_s)\|_{V^*}^{\frac{\alpha}{\alpha-1}}dsdt\Big]^{\frac{\alpha-1}{\alpha}}
\Big[\int_\delta^{T}\int_{t-\delta}^{t}\|\bar{X}^{\phi}_s-\bar{X}^{\phi}_{t-\delta}\|_{V}^{\alpha}dsdt\Big]^{\frac{1}{\alpha}}
\nonumber\\
\leq\!\!\!\!\!\!\!\!&&C\Big[\delta\int_0^{T}\big(1+\|\bar{X}^{\phi}_s\|_V^{\alpha}+\mu^0_s(\|\cdot\|_H^2)\big)ds\Big]^{\frac{\alpha-1}{\alpha}}\Big[\delta\int_0^{T}\|\bar{X}^{\phi}_s\|_V^{\alpha}ds\Big]^{\frac{1}{\alpha}}
\nonumber\\
\leq\!\!\!\!\!\!\!\!&&C_T\delta\Big[1+\sup_{t\in[0,T]}\|X^0_t\|_H^2+\int_0^{T}\|\bar{X}^{\phi}_s\|_V^{\alpha}ds\Big]^{\frac{\alpha-1}{\alpha}}\Big[\int_0^{T}\|\bar{X}^{\phi}_s\|_V^{\alpha}ds\Big]^{\frac{1}{\alpha}}
\nonumber\\
\leq\!\!\!\!\!\!\!\!&&C_{T,M}\delta(1+\|x\|_{H}^2+\sup_{t\in[0,T]}\|X^0_t\|_H^2),
\end{eqnarray}
where we used Young's inequality and Lemma \ref{l4} in the last step.

Similarly, we have
\begin{eqnarray}\label{17}
\!\!\!\!\!\!\!\!&&\int_\delta^{T}K_2(t)dt
\nonumber\\
\leq\!\!\!\!\!\!\!\!&&2\Big[\int_\delta^{T}\int_{t-\delta}^{t}\|B(s,\bar{X}^{\phi}_s,\mu^0_s)\|_{L_2(U,H)}^2\|\phi_s\|_U^2dsdt\Big]^{\frac{1}{2}}
\Big[\int_\delta^{T}\int_{t-\delta}^{t}\|\bar{X}^{\phi}_s-\bar{X}^{\phi}_{t-\delta}\|_H^2dsdt\Big]^{\frac{1}{2}}
\nonumber\\
\leq\!\!\!\!\!\!\!\!&&C\Big[\delta\int_0^{T}\big(1+\|\bar{X}^{\phi}_s\|_H^2+\mu^0_s(\|\cdot\|_H^2)\big)\|\phi_s\|_U^2ds\Big]^{\frac{1}{2}}
\Big[\delta\int_0^{T}\|\bar{X}^{\phi}_s\|_H^2ds\Big]^{\frac{1}{2}}
\nonumber\\
\leq\!\!\!\!\!\!\!\!&&C_T\delta\Big[\Big(1+\sup_{\psi\in S_M}\Big\{\sup_{t\in[0,T]}\|\bar{X}^{\psi}_t\|_H^2\Big\}+\sup_{t\in[0,T]}\|X^0_t\|_H^2\Big)\int_0^{T}\|\phi_s\|_U^2ds\Big]^{\frac{1}{2}}\Big[\int_0^{T}\|\bar{X}^{\phi}_s\|_H^2ds\Big]^{\frac{1}{2}}
\nonumber\\
\leq\!\!\!\!\!\!\!\!&&C_{T,M}\delta(1+\|x\|_{H}^2+\sup_{t\in[0,T]}\|X^0_t\|_H^2).
\end{eqnarray}
Substituting (\ref{16})-(\ref{17}) into (\ref{15}) gives that
\begin{eqnarray}\label{18}
\!\!\!\!\!\!\!\!&&\int_\delta^{T}\|\bar{X}^{\phi}_t-\bar{X}^{\phi}_{t-\delta}\|_{H}^2dt\leq C_{T,M}\delta(1+\|x\|_{H}^2+\sup_{t\in[0,T]}\|X^0_t\|_H^2).
\end{eqnarray}
Following the similar arguments as in (\ref{18}), one can easily infer that
\begin{eqnarray}\label{19}
\!\!\!\!\!\!\!\!&&\int_\delta^{T}\|\bar{X}^{\phi}_{t(\delta)}-\bar{X}^{\phi}_{t-\delta}\|_{H}^2dt\leq C_{T,M}\delta(1+\|x\|_{H}^2+\sup_{t\in[0,T]}\|X^0_t\|_H^2).
\end{eqnarray}
Finally, combining (\ref{18})-(\ref{19}) with (\ref{14}) implies Lemma \ref{l6}. \hspace{\fill}$\Box$
\end{proof}

After the above  preparations, we are in the position to verify the \textbf{Condition (A)} (ii).

\begin{theorem}\label{t3}
Suppose that $({\mathbf{H}}{\mathbf{1}})$-$({\mathbf{H}}{\mathbf{5}})$ hold. Let $\{\phi^n: n\in\mathbb{N}\}\subset S_M$ for any $M<\infty$ such that $\phi^n$ converges to element $\phi$ in $S_M$ as $n\to\infty$, then
$$\lim_{n\to\infty}\sup_{t\in[0,T]}\Big\|\mathcal{G}^0\Big(\int_0^t\phi^n_sds\Big)-\mathcal{G}^0\Big(\int_0^t\phi_sds\Big)\Big\|_H=0.$$
\end{theorem}
\begin{proof}
Let sequence $\{\bar{X}^{\phi^n}\}$ be the solution of Eq.~(\ref{e4}) with $\phi^n\in S_M$ replacing $\phi$, that is,
\begin{equation*}
\left\{ \begin{aligned}
&\frac{d\bar{X}^{\phi^n}_t}{dt}=A(t,\bar{X}^{\phi^n}_t,\mu^0_t)dt+B(t,\bar{X}^{\phi^n}_t,\mu^0_t)\phi^n_t,\\
&\bar{X}^{\phi^n}_0=x\in H.
\end{aligned} \right.
\end{equation*}

Let us denote $Z^{n}_t:=\bar{X}^{\phi^n}_t-\bar{X}^{\phi}_t$ satisfying
\begin{equation*}
\left\{ \begin{aligned}
&\frac{dZ^{n}_t}{dt}=A(t,\bar{X}^{\phi^n}_t,\mu^0_t)-A(t,\bar{X}^{\phi}_t,\mu^0_t)+B(t,\bar{X}^{\phi^n}_t,\mu^0_t)\phi^n_t-B(t,\bar{X}^{\phi}_t,\mu^0_t)\phi_t,\\
&Z^{n}_0=0.
\end{aligned} \right.
\end{equation*}
By the chain rule, we have
\begin{eqnarray}\label{20}
\|Z^{n}_t\|_H^2=\!\!\!\!\!\!\!\!&&2\int_0^t{}_{V^*}\langle A(s,\bar{X}^{\phi^n}_s,\mu^0_s)-A(s,\bar{X}^{\phi}_s,\mu^0_s),Z^{n}_s\rangle_Vds
\nonumber\\
\!\!\!\!\!\!\!\!&&+2\int_0^t\langle B(s,\bar{X}^{\phi^n}_s,\mu^0_s)\phi^n_s-B(s,\bar{X}^{\phi}_s,\mu^0_s)\phi_s,Z^{n}_s\rangle_Hds
\nonumber\\
=\!\!\!\!\!\!\!\!&&2\int_0^t{}_{V^*}\langle A(s,\bar{X}^{\phi^n}_s,\mu^0_s)-A(s,\bar{X}^{\phi}_s,\mu^0_s),Z^{n}_s\rangle_Vds
\nonumber\\
\!\!\!\!\!\!\!\!&&+2\int_0^t\langle \big(B(s,\bar{X}^{\phi^n}_s,\mu^0_s)-B(s,\bar{X}^{\phi}_s,\mu^0_s)\big)\phi^n_s,Z^{n}_s\rangle_Hds
\nonumber\\
\!\!\!\!\!\!\!\!&&+2\int_0^t\langle B(s,\bar{X}^{\phi}_s,\mu^0_s)(\phi^n_s-\phi_s),Z^{n}_s\rangle_Hds
\nonumber\\
=:\!\!\!\!\!\!\!\!&&\sum_{i=1}^3I_i(t).
\end{eqnarray}
We now aim to estimate terms $I_i$, $i=1,2,3$, respectively. Using Cauchy-Schwarz's inequality and ({\textbf{H}}{\textbf{3}}) leads to
\begin{eqnarray}\label{21}
\!\!\!\!\!\!\!\!&&I_1(t)+I_2(t)
\nonumber\\
\leq\!\!\!\!\!\!\!\!&&\int_0^t2{}_{V^*}\langle A(s,\bar{X}^{\phi^n}_s,\mu^0_s)-A(s,\bar{X}^{\phi}_s,\mu^0_s),Z^{n}_s\rangle_V+\|B(s,\bar{X}^{\phi^n}_s,\mu^0_s)-B(s,\bar{X}^{\phi}_s,\mu^0_s)\|_{L_2(U,H)}^2ds
\nonumber\\
\!\!\!\!\!\!\!\!&&+\int_0^t\|\phi^n_s\|_U^2\|Z^{n}_s\|_H^2ds
\nonumber\\
\leq\!\!\!\!\!\!\!\!&&C\int_0^t(1+\|\phi^n_s\|_U^2)\|Z^{n}_s\|_H^2ds.
\end{eqnarray}
From the fact $\phi^n\in S_M$ and Gronwall's inequality we obtain that
\begin{eqnarray}\label{27}
\sup_{t\in[0,T]}\|Z^{n}_t\|_H^2\leq C_{T,M}\Big[\sup_{t\in[0,T]}|I_3(t)|\Big].
\end{eqnarray}
It is clear that the term $\sup_{t\in[0,T]}|I_3(t)|$ can be controlled as follows
\begin{eqnarray}\label{22}
\sup_{t\in[0,T]}|I_3(t)|\leq\sum_{i=1}^5\widetilde{I}_i(n),
\end{eqnarray}
where
\begin{eqnarray*}
\widetilde{I}_1(n):=\!\!\!\!\!\!\!\!&&\sup_{t\in[0,T]}\Big|\int_0^{t}\langle B(s,\bar{X}^{\phi}_s,\mu^0_s)(\phi^n_s-\phi_s),Z^n_s-Z^n_{s(\delta)}\rangle_{H}ds\Big|,
\nonumber\\
\widetilde{I}_2(n):=\!\!\!\!\!\!\!\!&&\sup_{t\in[0,T]}\Big|\int_0^{t}\langle \big(B(s,\bar{X}^{\phi}_s,\mu^0_s)-B(s(\delta),\bar{X}^{\phi}_{s},\mu^0_s)\big)(\phi^n_s-\phi_s),Z^n_{s(\delta)}\rangle_{H}ds\Big|,
\nonumber\\
\widetilde{I}_3(n):=\!\!\!\!\!\!\!\!&&\sup_{t\in[0,T]}\Big|\int_0^{t}\langle \big(B(s(\delta),\bar{X}^{\phi}_s,\mu^0_s)-B(s(\delta),\bar{X}^{\phi}_{s(\delta)},\mu^0_{s(\delta)})\big)(\phi^n_s-\phi_s),Z^n_{s(\delta)}\rangle_{H}ds\Big|,
\nonumber\\
\widetilde{I}_4(n):=\!\!\!\!\!\!\!\!&&\sup_{t\in[0,T]}\Big|\int_{t(\delta)}^{t}\langle B(s(\delta),\bar{X}^{\phi}_{s(\delta)},\mu^0_{s(\delta)})(\phi^n_s-\phi_s),Z^n_{s(\delta)}\rangle_{H}ds\Big|,
\nonumber\\
\widetilde{I}_5(n):=\!\!\!\!\!\!\!\!&&\sum_{k=0}^{[T/\delta]-1}\Big|\langle B(k\delta,\bar{X}^{\phi}_{k\delta},\mu^0_{k\delta})\int_{k\delta}^{(k+1)\delta}(\phi^n_s-\phi_s)ds,Z^n_{k\delta}\rangle_{H}\Big|.
\end{eqnarray*}
Making use of Cauchy-Schwarz's inequality and H\"{o}lder's inequality, we have
\begin{eqnarray}\label{23}
\widetilde{I}_1(n)\leq\!\!\!\!\!\!\!\!&&\int_0^T\|B(t,\bar{X}^{\phi}_t,\mu^0_t)\|_{L_2(U,H)}\|\phi^n_t-\phi_t\|_U\|Z^n_t-Z^n_{t(\delta)}\|_Hdt
\nonumber\\
\leq\!\!\!\!\!\!\!\!&&\Big\{\int_0^T\|B(t,\bar{X}^{\phi}_t,\mu^0_t)\|_{L_2(U,H)}^2\|\phi^n_t-\phi_t\|_U^2dt\Big\}^{\frac{1}{2}}
\Big\{\int_0^T\|Z^n_t-Z^n_{t(\delta)}\|_H^2dt\Big\}^{\frac{1}{2}}
\nonumber\\
\leq\!\!\!\!\!\!\!\!&&C\Big\{\int_0^T\big(1+\|\bar{X}^{\phi}_t\|_H^2+\mu^0_t(\|\cdot\|_H^2)\big)\|\phi^n_t-\phi_t\|_U^2dt\Big\}^{\frac{1}{2}}
\nonumber\\
\!\!\!\!\!\!\!\!&&\times\Big\{\int_0^T(\|\bar{X}^{\phi^n}_t-\bar{X}^{\phi^n}_{t(\delta)}\|_H^2+\|\bar{X}^{\phi}_t-\bar{X}^{\phi}_{t(\delta)} \|_H^2)dt\Big\}^{\frac{1}{2}}
\nonumber\\
\leq\!\!\!\!\!\!\!\!&&C_{T,M}\Big\{\big(1+\sup_{\psi\in S_M}\{\sup_{t\in[0,T]}\|\bar{X}^{\psi}_t\|_H^2\}+\sup_{t\in[0,T]}\|X^0_t\|_H^2\big)\int_0^T\|\phi^n_t-\phi_t\|_U^2dt\Big\}^{\frac{1}{2}}
\nonumber\\
\!\!\!\!\!\!\!\!&&\times\Big\{\delta(1+\|x\|_{H}^2+\sup_{t\in[0,T]}\|X^0_t\|_H^2)\Big\}^{\frac{1}{2}}
\nonumber\\
\leq\!\!\!\!\!\!\!\!&&C_{T,M}\delta^{\frac{1}{2}}\big(1+\|x\|_{H}^2+\sup_{t\in[0,T]}\|X^0_t\|_H^2\big),
\end{eqnarray}
where we used Lemma \ref{l6} in the forth inequality, Lemma \ref{l4} in the last step and the fact that $\phi^n,\phi\in S_M$.

In terms of the condition ({\textbf{H}}{\textbf{5}}), it follows that
\begin{eqnarray}\label{24}
\widetilde{I}_2(n)\leq\!\!\!\!\!\!\!\!&&\int_0^T\|B(t,\bar{X}^{\phi}_t,\mu^0_t)-B(t(\delta),\bar{X}^{\phi}_{t},\mu^0_t)\|_{L_2(U,H)}\|\phi^n_t-\phi_t\|_U\|Z^n_{t(\delta)}\|_Hdt
\nonumber\\
\leq\!\!\!\!\!\!\!\!&&C\sqrt{\sup_{\psi\in S_M}\{\sup_{t\in[0,T]}\|\bar{X}^{\psi}_t\|_H^2\}}\int_0^T\delta^\gamma\Big(1+\|\bar{X}^{\phi}_{t}\|_H+\sqrt{\mu^0_t(\|\cdot\|_H^2)}\Big)\|\phi^n_t-\phi_t\|_Udt
\nonumber\\
\leq\!\!\!\!\!\!\!\!&&C_{T,M}\delta^{\gamma}\big(1+\|x\|_{H}^2+\sup_{t\in[0,T]}\|X^0_t\|_H^2\big).
\end{eqnarray}
Similar to (\ref{23}), using H\"{o}lder's inequality leads to
\begin{eqnarray*}
\widetilde{I}_3(n)\leq\!\!\!\!\!\!\!\!&&C\int_0^T\Big(\|\bar{X}^{\phi}_t-\bar{X}^{\phi}_{t(\delta)}\|_H+\mathbb{W}_{2,H}(\mu^0_t,\mu^0_{t(\delta)})\Big)\|\phi^n_t-\phi_t\|_U\|Z^n_{t(\delta)}\|_Hdt
\nonumber\\
\leq\!\!\!\!\!\!\!\!&&C\Big\{\int_0^T\|\bar{X}^{\phi^n}_{t(\delta)}-\bar{X}^{\phi}_{t(\delta)}\|_H^2\|\bar{X}^{\phi^n}_t-\bar{X}^{\phi^n}_{t(\delta)}\|_H^2dt\Big\}^{\frac{1}{2}}\Big\{\int_0^T\|\phi^n_t-\phi_t\|_U^2dt\Big\}^{\frac{1}{2}}
\nonumber\\
\!\!\!\!\!\!\!\!&&+C\Big\{\int_0^T\|\bar{X}^{\phi^n}_{t(\delta)}-\bar{X}^{\phi}_{t(\delta)}\|_H^2\mathbb{W}_{2,H}(\mu^0_t,\mu^0_{t(\delta)})^2dt\Big\}^{\frac{1}{2}}\Big\{\int_0^T\|\phi^n_t-\phi_t\|_U^2dt\Big\}^{\frac{1}{2}}
\nonumber\\
\leq\!\!\!\!\!\!\!\!&&C_M\sqrt{\sup_{\psi\in S_M}\{\sup_{t\in[0,T]}\|\bar{X}^{\psi}_t\|_H^2\}}\Big\{\int_0^T\|\bar{X}^{\phi^n}_t-\bar{X}^{\phi^n}_{t(\delta)}\|_H^2dt\Big\}^{\frac{1}{2}}
\nonumber\\
\!\!\!\!\!\!\!\!&&+C_M\sqrt{\sup_{\psi\in S_M}\{\sup_{t\in[0,T]}\|\bar{X}^{\psi}_t\|_H^2\}}\Big\{\int_0^T\|X^0_t-X^0_{t(\delta)}\|_H^2dt\Big\}^{\frac{1}{2}}.
\end{eqnarray*}
Note that $X^0\in C([0,T];H)$, without loss of generality, we assume that $\|X^0_t-X^0_{t(\delta)}\|_H\leq\delta$. By Lemma \ref{l4} and \ref{l6}, it is clear that
\begin{eqnarray}\label{25}
\widetilde{I}_3(n)\leq\!\!\!\!\!\!\!\!&&C_{T,M}\delta^{\frac{1}{2}}\big(1+\|x\|_{H}^2+\sup_{t\in[0,T]}\|X^0_t\|_H^2\big).
\end{eqnarray}
By H\"{o}lder's inequality, one can infer that
\begin{eqnarray}\label{26}
\widetilde{I}_4(n)\leq\!\!\!\!\!\!\!\!&&\sqrt{\sup_{\psi\in S_M}\{\sup_{t\in[0,T]}\|\bar{X}^{\psi}_t\|_H^2\}}\Big\{\sup_{t\in[0,T]}\Big|\int_{t(\delta)}^{t} \|B(s(\delta),\bar{X}^{\phi}_{s(\delta)},\mu^0_{s(\delta)})\|_{L_2(U,H)}\|\phi^n_s-\phi_s\|_Uds\Big|\Big\}
\nonumber\\
\leq\!\!\!\!\!\!\!\!&&\delta^{\frac{1}{2}}\sqrt{\sup_{\psi\in S_M}\{\sup_{t\in[0,T]}\|\bar{X}^{\psi}_t\|_H^2\}}\Big\{\sup_{t\in[0,T]}\Big|\int_{t(\delta)}^{t} \|B(s(\delta),\bar{X}^{\phi}_{s(\delta)},\mu^0_{s(\delta)})\|_{L_2(U,H)}^2\|\phi^n_s-\phi_s\|_U^2ds\Big|\Big\}^{\frac{1}{2}}
\nonumber\\
\leq\!\!\!\!\!\!\!\!&&C\delta^{\frac{1}{2}}\sqrt{\sup_{\psi\in S_M}\{\sup_{t\in[0,T]}\|\bar{X}^{\psi}_t\|_H^2\}}\Big\{\int_0^T\big(1+\|\bar{X}^{\phi}_{t(\delta)}\|_H^2+\mu^0_{t(\delta)}(\|\cdot\|_H^2)\big)\|\phi^n_t-\phi_t\|_U^2dt\Big\}^{\frac{1}{2}}
\nonumber\\
\leq\!\!\!\!\!\!\!\!&&C_{T,M}\delta^{\frac{1}{2}}\big(1+\|x\|_{H}^2+\sup_{t\in[0,T]}\|X^0_t\|_H^2\big).
\end{eqnarray}

Note that $\phi^n\to\phi$ in the weak topology of $S_M$, thus for any $a,b\in[0,T]$,~$a<b$, the integral $\int_a^b\phi_s^n ds\to\int_a^b\phi_sds$ weakly in $U$. For the term $\widetilde{I}_5(n)$, since $B(k\delta,\bar{X}^{\phi}_{k\delta},\mu^0_{k\delta})$ is a compact operator, the sequence $B(k\delta,\bar{X}^{\phi}_{k\delta},\mu^0_{k\delta})\int_{k\delta}^{(k+1)\delta}(\phi^n_s-\phi_s)ds$ strongly converges to $0$ in $H$ for any fixed $k$, as $n\to\infty$. Consequently, the boundedness of $Z^n_{k\delta}$ implies that
\begin{eqnarray}\label{28}
\lim_{n\to\infty}\widetilde{I}_5(n)=0.
\end{eqnarray}
Finally, combining (\ref{27})-(\ref{28}) and taking $\delta\to0$, one can conclude that
$$\lim_{n\to\infty}\sup_{t\in[0,T]}\Big\|\mathcal{G}^0\Big(\int_0^t\phi^n_sds\Big)-\mathcal{G}^0\Big(\int_0^t\phi_sds\Big)\Big\|_H=\lim_{n\to\infty}\sup_{t\in[0,T]}\|Z_t^{n}\|_H=0.$$

The proof is completed.  \hspace{\fill}$\Box$
\end{proof}

Now we are able to complete the proof of our main result of Theorem \ref{t1}.

\vspace{1mm}
\textbf{Proof of Theorem \ref{t1}:} Combining Theorem \ref{t2} with Theorem \ref{t3}, we know that Lemma \ref{l3} implies that $\{X^{\varepsilon}\}$ satisfies the Laplace principle, which is equivalent to the LDP on $C([0,T]; H)$ with a good rate function $I$ defined in (\ref{rf}). \hspace{\fill}$\Box$

\section{Application to examples}
\setcounter{equation}{0}
 \setcounter{definition}{0}
This section is devoted to applying the general result obtained in Theorem \ref{t1} to establish the LDP result for several concrete McKean-Vlasov type SPDEs.

In the sequel, we always assume $\Lambda\subset\mathbb{R}^d$ as a bounded domain with smooth boundary $\partial\Lambda$. Denote by
$C_0^\infty(\Lambda, \mathbb{R}^d)$ the space of all smooth functions from $\Lambda$ to $\mathbb{R}^d$ with compact support.  For each $r\ge 1$, let $L^r(\Lambda, \mathbb{R}^d)$ be the vector valued $L^r$-space endowed with the norm $\|\cdot\|_{L^r}$.
For any integer $m>0$, let us denote by $W_0^{m,r}(\Lambda, \mathbb{R}^d)$ the classical Sobolev space (with Dirichlet boundary condition) from domain $\Lambda$
to $\mathbb{R}^d$ with the following (equivalent) norm:
$$ \|u\|_{W^{m,r}} = \left( \sum_{ |\alpha|= m} \int_{\Lambda} |D^\alpha u|^rd x \right)^\frac{1}{r}.$$

\subsection{McKean-Vlasov stochastic porous media equation}
The first example is the McKean-Vlasov stochastic porous media equation, which is arisen originally as a model for gas flow in a porous medium (cf. e.g.\cite{LR1}).

For any $r\geq 2$, we set the following Gelfand triple
$$V:=L^r(\Lambda)\subset(W_0^{1,2}(\Lambda))^*=:H\subset(L^r(\Lambda))^*=V^*,$$
and consider equation
\begin{equation}\label{e9}
dX^\varepsilon_t=\Delta\Psi(t,X^\varepsilon_t,\mathscr{L}_{X^\varepsilon_t})dt+\sqrt{\varepsilon}B(t,X^\varepsilon_t,\mathscr{L}_{X^\varepsilon_t})dW_t,~X^\varepsilon_0=x\in H,
\end{equation}
where $W_t$ is a cylindrical Wiener process defined
on a probability space $(\Omega,\mathscr {F},\mathscr
{F}_t,\mathbb{P})$ taking values in a sparable Hilbert space $U$, maps
$$\Psi:[0,T]\times V\times\mathscr{P}(H)\to L^{\frac{r}{r-1}}(\Lambda)$$
and
$$B:[0,T]\times V\times\mathscr{P}(H)\to L_2(U,H)$$
are some measurable maps.

We recall an important lemma as follows (see e.g.\cite[Lemma 4.1.13]{LR1}).
\begin{lemma}\label{l7}
The map
$$\Delta:W_0^{1,2}(\Lambda)\to (L^r(\Lambda))^*$$
could be extend to a linear isometry
$$\Delta:L^{\frac{r}{r-1}}(\Lambda)\to (L^r(\Lambda))^*.$$
Furthermore, for any $u\in L^{\frac{r}{r-1}}(\Lambda)$, $v\in L^r(\Lambda)$
$$_{V^*}\langle-\Delta u,v\rangle_V=_{L^{\frac{r}{r-1}}}\langle u,v\rangle_{L^r}=\int_\Lambda u(\xi)v(\xi)d\xi.$$
\end{lemma}

We formulate the assumptions on $\Psi$.
\begin{hypothesis}\label{h2}
There are some constants $\theta,C>0$ such that the following conditions hold.
\begin{enumerate}
\item [$(\Psi1)$] For all $t\in[0,T]$, $v\in V$, the map
\begin{eqnarray*}
V\times\mathscr{P}_2(H)\ni(u,\mu)\mapsto\int_{\Lambda}\Psi(t,u,\mu)(\xi)v(\xi)d\xi
\end{eqnarray*}
is continuous.
\item [$(\Psi2)$] For all $t\in[0,T]$, $u\in V$ and $\mu\in\mathscr{P}_2(H)$,
\begin{eqnarray*}
\int_{\Lambda}\Psi(t,u,\mu)(\xi)u(\xi)d\xi\geq -C\big(1+\|u\|_H^2+\mu(\|\cdot\|_H^2)\big)+\theta\|u\|_V^r.
\end{eqnarray*}
\item [$(\Psi3)$]For all $t\in[0,T]$, $u,v\in V$ and $\mu,\nu\in\mathscr{P}_2(H)$,
\begin{eqnarray*}
\int_{\Lambda}\big(\Psi(t,u,\mu)(\xi)-\Psi(t,v,\nu)(\xi)\big)\big(u(\xi)-v(\xi)\big)d\xi\geq 0.
\end{eqnarray*}
\item [$(\Psi4)$] For all $t\in[0,T]$, $u\in V$ and $\mu\in\mathscr{P}_2(H)$,
\begin{eqnarray*}
\|\Psi(t,u,\mu)\|_{L^{\frac{r}{r-1}}}^{\frac{r}{r-1}}\leq C\big(1+\|u\|_V^{r}+\mu(\|\cdot\|_H^2)\big).
\end{eqnarray*}
\end{enumerate}
\end{hypothesis}
After the preparations above, we now define map $A:[0,T]\times V\times\mathscr{P}(H)\to V^*$ by
$$A(t,u,\mu):=\Delta\Psi(t,u,\mu).$$
In terms of Lemma \ref{l7}, it is easy to see that $A$ is well-defined and takes value in $V^*$. Moreover, it is clear that the conditions $(\Psi1)$-$(\Psi4)$ imply
$({\mathbf{H}}{\mathbf{1}})$-$({\mathbf{H}}{\mathbf{4}})$. In order to prove the LDP, we further assume that there are some constants $C,\gamma>0$ such that for any $t,s\in[0,T]$, $u,v\in V$ and $\mu,\nu\in\mathscr{P}_2(H)$,
\begin{equation}\label{29}
\|B(t,u,\mu)-B(t,v,\nu)\|_{L_2(U,H)}^2\leq C\big(\|u-v\|_H^2+\mathbb{W}_{2,H}(\mu,\nu)^2\big),
\end{equation}
\begin{equation}\label{30}
\|B(t,u,\mu)-B(s,u,\mu)\|_{L_2(U,H)}\leq C\Big(1+\|u\|_H+\sqrt{\mu(\|\cdot\|_H^2)}\Big)|t-s|^\gamma
\end{equation}
and
\begin{equation}\label{32}
\int_0^T\|B(t,0,\delta_0)\|_{L_2(U,H)}^2dt<\infty.
\end{equation}
Hence we are in the position to derive the following LDP result for the distribution dependent stochastic porous media equation.
\begin{theorem}\label{SPME}
Assume that (\ref{29})-(\ref{32}) hold and $\Psi$ fulfills the conditions $(\Psi1)$-$(\Psi4)$ above,
then  $\{X^{\varepsilon}:\varepsilon>0\}$ in (\ref{e9})
satisfies the LDP on $C([0,T]; H)$ with the
good rate function $I$ given by $(\ref{rf})$.
\end{theorem}

\begin{Rem}
In \cite{HL}, the authors have established the strong/weak existence and uniqueness of solutions for more general type of distribution dependent stochastic porous media equations.
 To the best of our knowledge, there is no LDP result in the literature obtained for McKean-Vlasov type quasilinear SPDE such as stochastic porous media equations.
\end{Rem}

\subsection{McKean-Vlasov stochastic $p$-Laplace equation}\label{laplace}
Now we  apply our main result the establish the LDP for the following type McKean-Vlasov stochastic $p$-Laplace equation
\begin{equation}\label{e10}
\left\{ \begin{aligned}
&dX^\varepsilon_t=\Big[div(|\nabla X^\varepsilon_t|^{p-2}\nabla X^\varepsilon_t)+F(t,X^\varepsilon_t,\mathscr{L}_{X^\varepsilon_t})\Big]dt+\sqrt{\varepsilon}B(t,X^\varepsilon_t,\mathscr{L}_{X^\varepsilon_t})dW_t,\\
&X^\varepsilon_0=x\in H,
\end{aligned} \right.
\end{equation}
where $p\geq2$ and $W_t$ is a cylindrical Wiener process defined
on a probability space $(\Omega,\mathscr {F},\mathscr{F}_t,\mathbb{P})$ taking values in $U$.

We set the Gelfand triple as follows
$$V:=W_0^{1,p}(\Lambda)\subset H:=L^2(\Lambda)\subset(W_0^{1,p}(\Lambda))^*=V^*.$$
Suppose that  there is a constant $C>0$ such that for all $t\in[0,T]$, $u,v\in V$ and $\mu,\nu\in\mathscr{P}_2(H)$, the map $F:[0,T]\times V\times\mathscr{P}(H)\to H$ satisfies
\begin{eqnarray}\label{31}
\|F(t,u,\mu)-F(t,v,\nu)\|_H\leq C\big(\|u-v\|_H+\mathbb{W}_{2,H}(\mu,\nu)\big).
\end{eqnarray}

Denote $\bar{A}(u):=div(|\nabla u|^{p-2}\nabla u)$, which is called $p\text{-}Laplacian$ operator.
It is well-known that the operator $\bar{A}+F$ satisfies $({\mathbf{H}}{\mathbf{1}})$-$({\mathbf{H}}{\mathbf{4}})$, interested readers can refer to e.g.~\cite[Example 4.1.9]{LR1} for the detailed proof. Thus, according to Theorem \ref{t1} we have the following LDP result for distribution dependent stochastic $p$-Laplace equations.
\begin{theorem}\label{laplace}
Assume that (\ref{29})-(\ref{32}) and (\ref{31}) hold,
then  $\{X^{\varepsilon}:\varepsilon>0\}$ in (\ref{e10})
satisfies the LDP on $C([0,T]; H)$ with the
good rate function $I$ given by $(\ref{rf})$.
\end{theorem}

\begin{Rem}
The Freidlin-Wentzell type LDP for stochastic $p$-Laplace equation has been established in the works \cite{L1,RZ2}. In this work, we are able to extend the corresponding LDP result to the distribution dependent case and also impose weaker assumptions. In particular, if we take $p=2$, $\bar{A}$ reduces to the classical Laplace operator. Therefore, our result above also covers a class of distribution dependent semilinear SPDEs.
\end{Rem}

\subsection{McKean-Vlasov stochastic differential equations}
Our main result is also applicable to finite dimensional McKean-Vlasov SDE, if we consider $V=H=\mathbb{R}^d$ with the Euclidean norm $|\cdot|$ and inner product $\langle\cdot,\cdot\rangle$,
\begin{equation}\label{e11}
\left\{ \begin{aligned}
&dX_t^\varepsilon=b(t,X_t^\varepsilon,\mathscr{L}_{X_t^\varepsilon})dt+\sqrt{\varepsilon}\sigma(t,X_t^\varepsilon,\mathscr{L}_{X_t^\varepsilon})dW_t,\\
&X_0^\varepsilon=x\in\mathbb{R}^d,
\end{aligned} \right.
\end{equation}
where $W_t$ is a standard $d$-dimensional Brownian motion defined
on a probability space $(\Omega,\mathscr {F},\mathscr{F}_t,\mathbb{P})$. Suppose the coefficients
$$b:[0,T]\times \mathbb{R}^d\times\mathscr{P}(\mathbb{R}^d)\rightarrow \mathbb{R}^d,~\sigma:[0,T]\times \mathbb{R}^d\times\mathscr{P}(\mathbb{R}^d)\rightarrow \mathbb{R}^{d\times d}$$
 are measurable and satisfy the following conditions, where $\mathbb{R}^{d\times d}$ denotes the set of real $d\times d$ matrices.
\begin{hypothesis}\label{h3}
There are some constants $\gamma,C>0$ such that the following conditions hold.
\begin{enumerate}
\item [$({\mathbf{A}}{\mathbf{1}})$] (Continuity) For all $t\in[0,T]$, the map
\begin{eqnarray*}
\mathbb{R}^d\times\mathscr{P}_2(\mathbb{R}^d)\ni(u,\mu)\mapsto b(t,u,\mu)
\end{eqnarray*}
is continuous.
\item [$({\mathbf{A}}{\mathbf{2}})$] (Monotonicity) For all $t\in[0,T]$, $u,v\in \mathbb{R}^d$ and $\mu,\nu\in\mathscr{P}_2(\mathbb{R}^d)$,
\begin{eqnarray*}
\langle b(t,u,\mu)-b(t,v,\nu),u-v\rangle \leq C\big(|u-v|^2+\mathbb{W}_{2,\mathbb{R}^d}(\mu,\nu)^2\big).
\end{eqnarray*}
\item [$({\mathbf{A}}{\mathbf{3}})$] (Growth) For all $t\in[0,T]$, $u\in \mathbb{R}^d$ and $\mu\in\mathscr{P}_2(\mathbb{R}^d)$,
\begin{eqnarray*}
|b(t,u,\mu)|\leq C\big(1+|u|+\mu(|\cdot|^2)^\frac{1}{2}\big).
\end{eqnarray*}
\item [$({\mathbf{A}}{\mathbf{4}})$] For all $t,s\in[0,T]$, $u,v\in \mathbb{R}^d$ and $\mu,\nu\in\mathscr{P}_2(\mathbb{R}^d)$,
\begin{equation*}
\|\sigma(t,u,\mu)-\sigma(t,v,\nu)\|^2\leq C\big(|u-v|^2+\mathbb{W}_{2,\mathbb{R}^d}(\mu,\nu)^2\big),
\end{equation*}
\begin{equation*}
\|\sigma(t,u,\mu)-\sigma(s,u,\mu)\|\leq C\Big(1+|u|+\sqrt{\mu(|\cdot|^2)}\Big)|t-s|^\gamma
\end{equation*}
and
\begin{equation*}
\int_0^T\|\sigma(t,0,\delta_0)\|^2dt<\infty,
\end{equation*}
where $\|\cdot\|$ denotes the matrix norm.
\end{enumerate}
\end{hypothesis}

Now we are in the position to derive the LDP for MVSDEs.
\begin{theorem}\label{SDE}
Assume that  Hypothesis \ref{h3} hold,
then $\{X^{\varepsilon}:\varepsilon>0\}$ in (\ref{e11})
satisfies the LDP on $C([0,T];\mathbb{R}^d)$ with the
good rate function $I$ given by $(\ref{rf})$.
\end{theorem}

\begin{Rem}
Dos Reis et al.~\cite{DST} used classical time discretization and exponential equivalence arguments to investigate the Freidlin-Wentzell's LDP for a class of MVSDEs under the monotonicity and locally Lipschitz conditions with respect to the solution and global Lipschitz with respect to the distribution. In this work, based on the weak convergence method and generalized variational framework, we give a more succinct proof for the LDP associated with some MVSDEs under monotonicity and linear growth assumptions.
\end{Rem}

\noindent\textbf{Acknowledgements} {The authors would like to thank the referees  for their very constructive  suggestions and valuable  comments. The research
of S. Li is supported by NSFC (No.~12001247),
 NSF of Jiangsu Province (No.~BK20201019),  NSF of Jiangsu Higher Education Institutions of China (No. 20KJB110015)
and the Foundation of Jiangsu Normal University (No.~19XSRX023). The research of W. Liu is supported by NSFC (No.~11822106, 11831014, 12090011) and the PAPD of Jiangsu Higher Education Institutions.}


\begin{thebibliography}{2}
\bibitem{A} R. Azencott, \emph{Grandes d\'{e}viations et
applications},  Eighth Saint Flour Probability Summer School-1978 (Saint Flour, 1978), Lecture Notes in Math., \textbf{774}, Springer, Berlin (1980), 1--176.
\bibitem{BCD} A. Budhiraja, J. Chen, P. Dupuis, \emph{Large deviations for stochastic partial differential equations driven by a Poisson random measure},
Stochastic Process. Appl. \textbf{123} (2013), 523--560.
\bibitem{BD} A. Budhiraja, P. Dupuis, \emph{A variational
 representation for positive functionals
of infinite dimensional Brownian motion},  Probab. Math. Statist.
 \textbf{20} (2000), 39--61.
\bibitem{BDM} A. Budhiraja, P. Dupuis, V. Maroulas, \emph{Large deviations for infinite dimensional stochastic
dynamical systems}, Ann. Probab.  \textbf{36} (2008), 1390--1420.
\bibitem{BGJ} Z. Brze\'{z}niak, B. Goldys, T. Jegaraj,
\emph{Large deviations and transitions between equilibria for stochastic Landau-Lifshitz-Gilbert equation},
Arch. Ration. Mech. Anal. \textbf{226} (2017), 497--558.
\bibitem{BLPR} R. Buckdahn, J. Li, S. Peng, C. Rainer,  \emph{Mean-field stochastic differential equations and associated PDEs}, Ann. Probab. \textbf{45} (2017), 824--878.
\bibitem{BM} H. Bessaih, A. Millet, \emph{Large deviation principle and inviscid shell models}, Electron. J.
Probab. \textbf{14} (2009), 2551--2579.


\bibitem{BR0} V. Barbu, M. R\"{o}ckner, \emph{Solutions for nonlinear Fokker-Planck equations with measures as initial data and McKean-Vlasov equations}, J. Funct. Anal. \textbf{280} (2021),  108926.
\bibitem{BR1} V. Barbu, M. R\"{o}ckner, \emph{From non-linear Fokker-Planck equations to solutions of distribution dependent SDE},  Ann. Probab. \textbf{48} (2020),  1902--1920.
\bibitem{BR2} V. Barbu, M. R\"{o}ckner, \emph{Probabilistic representation for solutions to non-linear Fokker-Planck equations}, SIAM J. Math. Anal. \textbf{50} (2018), 4246--4260.



\bibitem{C1} H. Comman, \emph{Criteria for large deviations}, Trans. Amer. Math. Soc. \textbf{355} (2003), 2905--2923.
\bibitem{CR} S. Cerrai, M. R\"ockner, \emph{Large deviations
for stochastic reaction-diffusion systems with multiplicative noise
and non-Lipschitz reaction term}, Ann. Probab. \textbf{32} (2004),
1100--1139.
\bibitem{CG} Y. Chen, H. Gao,
\emph{Well-posedness and large deviations for a class of SPDEs with L\'{e}vy noise},
J. Differential Equations \textbf{263} (2017), 5216--5252.
\bibitem{CM} I. Chueshov, A. Millet, \emph{Stochastic 2D hydrodynamical type systems: well posedness and large deviations}, Appl. Math. Optim. \textbf{61} (2010), 379--420.


\bibitem{DE} P. Dupuis, R. Ellis, \emph{A weak convergence
 approach to the theory of large deviations}, Wiley, New York. 1997.


\bibitem{DaZa} G. Da Prato, J. Zabczyk, \emph{Stochastic Equations
 in Infinite Dimensions}, Encyclopedia of Mathematics and its Applications, Cambridge University Press. 1992.
\bibitem{DST} G. Dos Reis, W. Salkeld, J. Tugaut, \emph{Freidlin-Wentzell LDPs in path space for McKean-Vlasov equations and the functional iterated logarithm law}, Ann. Appl. Probab. \textbf{29} (2019),  1487--1540.
\bibitem{DWZZ} Z. Dong, J.-L. Wu, R. Zhang, T. Zhang, \emph{Large deviation principles for first-order scalar conservation laws with stochastic forcing}, Ann. Appl. Probab. \textbf{30} (2020), 324--367.
\bibitem{DZ} A. Dembo, O. Zeitouni, \emph{Large Deviations Techniques and Applications}, Springer-Verlag, New York, 2000.


\bibitem{Fr} M.I. Freidlin, \emph{Random perturbations of
 reaction-diffusion equations: the quasi-deterministic approximations,}
 Trans. Amer.
 Math. Soc. \textbf{305} (1988), 665--697.
\bibitem{FW} M.I. Freidlin, A.D. Wentzell, \emph{Random
 perturbations of dynamical systems,} Translated from the Russian by
 Joseph Sz\"{u}cs.
 Grundlehren der Mathematischen Wissenschaften [Fundamental Principles
 of Mathematical Sciences], 260. Springer-Verlag, New York, 1984.

\bibitem{HIP} S. Herrmann, P. Imkeller, D. Peithmann, \emph{Large deviations and a Kramers' type law for selfstabilizing diffusions}, Ann. Appl. Probab. \textbf{18} (2008),  1379--1423.

\bibitem{HL} W. Hong, W. Liu, \emph{Distribution Dependent Stochastic Porous Media Type
Equations on General Measure Spaces}, arXiv:2103.10135v1.
\bibitem{HRW} X. Huang, P. Ren, F.-Y. Wang, \emph{Distribution Dependent Stochastic Differential Equations}, {\it Front. Math. China} \textbf{16} (2021), 257--301.
\bibitem{HRW1} X. Huang, M. R\"{o}ckner, F.-Y. Wang, \emph{Nonlinear Fokker-Planck equations for probability measures on path space and path-distribution dependent SDEs}, Discrete Contin. Dyn. Syst. \textbf{39} (2019), 3017--3035.
\bibitem{HS} X. Huang, Y. Song, \emph{Well-posedness and regularity for distribution dependent SPDEs with singular drifts}, Nonlinear Anal. \textbf{203} (2021), 112167.
\bibitem{HW1} X. Huang, F.-Y. Wang, \emph{Distribution dependent SDEs with singular coefficients}, Stochastic Process. Appl. \textbf{129} (2019), 4747--4770.
\bibitem{M1} H.P. McKean, \emph{Propagation of chaos for a class of nonlinear parabolic equations}, Lecture Series in
Differential Equations, \textbf{7} (1967), 41--57.



\bibitem{K1} M. Kac, \emph{Foundations of kinetic theory}, In: Proceedings of the Third Berkeley Symposium on Mathematical Statistics and Probability,  1954--1955, Vol. III, pp. 171--197. University California Press,
Berkeley, 1956.
 \bibitem{K2} R.Z. Khasminskii, \emph{On an averging principle for It\^{o} stochastic differential equations}, Kibernetica  \textbf{4}
(1968), 260--279.


\bibitem{L1} W. Liu, \emph{Large deviations for stochastic evolution equations with small multiplicative noise}, Appl. Math. Optim. \textbf{61} (2010), 27--56.
\bibitem{LR1} W. Liu, M. R\"{o}ckner, \emph{Stochastic Partial Differential Equations: An Introduction}, Universitext, Springer, 2015.
\bibitem{LSZZ} W. Liu, Y. Song, J. Zhai, T. Zhang, \emph{Large and moderate deviation principles for McKean-Vlasov SDEs with jumps}, arXiv:2011.08403v1.

\bibitem{MSS} U. Manna, S.S. Sritharan, P. Sundar, \emph{Large deviations for the stochastic shell model of
turbulence}, NoDEA Nonlinear Differential Equations Appl. \textbf{16} (2009),
493--521.
\bibitem{MSZ} A. Matoussi, W. Sabbagh, T. Zhang, \emph{Large deviation principles of obstacle problems for quasilinear stochastic
PDEs}, Appl. Math. Optim.  \textbf{83} (2021), 849--879.

\bibitem{P} S. Peszat, \emph{Large deviation principle for stochastic
evolution equations,} Probab. Theory Relat. Fields. \textbf{98} (1994),
113--136.

\bibitem{RTW} P. Ren, H. Tang, F.-Y. Wang, \emph{Distribution-Path Dependent Nonlinear SPDEs with Application to Stochastic Transport Type Equations}, arXiv:2007.09188v3.
\bibitem{RW2} P. Ren, F.-Y. Wang, \emph{Donsker-Varadhan Large Deviations for Path-Distribution Dependent SPDEs},  J. Math. Anal. Appl.  \textbf{499} (2021), 125000.
\bibitem{RW3} P. Ren, F.-Y. Wang, \emph{Bismut formula for Lions derivative of distribution dependent SDEs and applications}, J. Differential Equations \textbf{267} (2019), 4745--4777.
\bibitem{RWW} M. R\"ockner, F.-Y. Wang, L. Wu, \emph{Large
 deviations for stochastic generalized porous media equations}, Stochastic Process. Appl. \textbf{116} (2006), 1677--1689.
 \bibitem{RZ12} M. R\"{o}ckner,  T.S. Zhang, X. Zhang \emph{Large deviations for stochastic tamed 3D Navier-Stokes equations}, Appl. Math. Optim. \textbf{61}  (2010), 267--285.
 \bibitem{RZ2} J. Ren, X. Zhang, \emph{Freidlin-Wentzell's Large
 Deviations for Stochastic Evolution Equations}, J. Funct. Anal.  \textbf{254} (2008), 3148--3172.

\bibitem{S} D.W. Stroock, \emph{An Introduction to the Theory of Large
 Deviations}, Spring-Verlag, New York, 1984.
\bibitem{SS} S.S. Sritharan, P. Sundar, \emph{Large deviations for
 the two-dimensional Navier-Stokes
 equations with multiplicative noise}, Stochastic Process. Appl.  \textbf{116} (2006),
1636--1659.

\bibitem{V1} S.R.S. Varadhan, \emph{Large deviations and Applications}, \textbf{46}, CBMS-NSF Series in Applied Mathematics, SIAM, Philadelphia, 1984.

\bibitem{W1} F.-Y. Wang, \emph{Distribution dependent SDEs for Landau type equations}, Stochastic Process. Appl. \textbf{128} (2018), 595--621.

\bibitem{XZ} J. Xiong, J. Zhai, \emph{Large deviations for locally monotone stochastic partial differential equations driven by L\'{e}vy noise}, Bernoulli \textbf{24} (2018), 2842--2874.

\bibitem{Z} E. Zeidler, \emph{Nonlinear Functional Analysis and
its Applications, II/B, Nonlinear Monotone Operators},
Springer-Verlag, New York, 1990.
\end{thebibliography}
\end{document}